\documentclass[reqno,12pt]{amsart}

\usepackage[utf8]{inputenc}
\usepackage{microtype}
\usepackage{amsfonts}
\usepackage{amsmath}
\usepackage{graphicx}
\usepackage{float}
\usepackage{color}
\usepackage{hyperref}
\usepackage{tikz-cd}

\usepackage{amssymb}
\usepackage{enumitem}
\usepackage{mathrsfs}

\usepackage{tikz}
\usetikzlibrary{topaths}
\usetikzlibrary{calc}

\usepackage{a4wide}

\usepackage{empheq}

\newtheorem{theorem}{Theorem}[section]
\newtheorem*{theorem*}{Theorem}
\newtheorem{lemma}[theorem]{Lemma}
\newtheorem{proposition}[theorem]{Proposition}
\newtheorem*{proposition*}{Proposition}
\newtheorem{corollary}[theorem]{Corollary}
\newtheorem*{corollary*}{Corollary}

\newtheorem*{conjecture*}{Conjecture}
\newtheorem{question}[theorem]{Question}
\newtheorem*{question*}{Question}

\newtheorem*{main:contain_symfin}{Theorem~\ref{thrm:contain_symfin}}
\newtheorem*{main:embed_fin_gen}{Theorem~\ref{thrm:embed_fin_gen}}
\newtheorem*{main:amenable}{Theorem~\ref{thrm:amenable}}

\theoremstyle{definition}
\newtheorem{definition}[theorem]{Definition}
\newtheorem{remark}[theorem]{Remark}
\newtheorem{example}[theorem]{Example}
\newtheorem{observation}[theorem]{Observation}

\newcommand{\Z}{\mathbb{Z}}
\newcommand{\N}{\mathbb{N}}

\newcommand{\shift}{\psi}
\newcommand{\slide}{\alpha}
\newcommand{\flip}{\lambda}
\newcommand{\tree}{\mathcal{T}}
\newcommand{\state}{\phi}

\newcommand{\Germs}{\mathcal{G}}

\DeclareMathOperator{\Aut}{Aut}

\DeclareMathOperator{\Sym}{Sym}
\DeclareMathOperator{\Symfin}{Sym_{fin}}

\DeclareMathOperator{\F}{F}

\DeclareMathOperator{\Stab}{Stab}

\DeclareMathOperator{\id}{id}

\numberwithin{equation}{section}

\begin{document}

\title{Houghton-like groups from ``shift-similar'' groups}
\date{\today}


\author[B.~Mallery]{Brendan Mallery}
\address{Department of Mathematics, Tufts University, Medford, MA 02155}
\email{brendan.mallery@tufts.edu}

\author[M.~C.~B.~Zaremsky]{Matthew C.~B.~Zaremsky}
\address{Department of Mathematics and Statistics, University at Albany (SUNY), Albany, NY 12222}
\email{mzaremsky@albany.edu}

\begin{abstract}
We introduce and study \emph{shift-similar} groups $G\le\Sym(\N)$, which play an analogous role in the world of Houghton groups that self-similar groups play in the world of Thompson groups. We also introduce Houghton-like groups $H_n(G)$ arising from shift-similar groups $G$, which are an analog of R\"over--Nekrashevych groups from the world of Thompson groups. We prove a variety of results about shift-similar groups and these Houghton-like groups, including results about finite generation and amenability. One prominent result is that every finitely generated group embeds as a subgroup of a finitely generated shift-similar group, in contrast to self-similar groups, where this is not the case. This establishes in particular that there exist uncountably many isomorphism classes of finitely generated shift-similar groups, again in contrast to the self-similar situation.
\end{abstract}

\maketitle
\thispagestyle{empty}

\section*{Introduction}

In this paper, we introduce and study a new family of groups that we call \emph{shift-similar} groups. These are an analog of the well known family of self-similar groups, but instead of arising from actions on trees they arise from actions on the set of natural numbers $\N$. Where self-similar groups $G$ interact nicely with Higman--Thompson groups $V_d$ to produce R\"over--Nekrashevych groups $V_d(G)$, our shift-similar groups $G$ interact nicely with Houghton groups $H_n$ to produce another new family of groups $H_n(G)$. Thus, our construction provides an analogy in the world of Houghton groups for a number of important concepts from the world of Thompson groups.

A self-similar group $G$ is a group of automorphisms of a locally finite, regular, rooted tree, satisfying a certain rule that ensures that any element of $G$ still looks like an element of $G$ when restricted to certain canonical copies of the tree inside of itself. See Definition~\ref{def:self_similar} for more precision, and \cite{nekrashevych05} for a wealth of background. Under our new definition, a shift-similar group $G$ is a group of permutations of $\N$, satisfying a certain rule that ensures that any element of $G$ still looks like an element of $G$ when restricted to certain canonical copies of $\N$ inside of itself (see Definition~\ref{def:shift_similar} for more precision). The point is that both regular rooted trees and $\N$ contain canonical copies of themselves, and these groups are in some sense closed under restricting to these copies.

Self-similar groups $G$ naturally lead to so called R\"over--Nekrashevych groups $V_d(G)$. If $G$ is a self-similar group of automorphisms of the rooted $d$-ary tree, then $V_d(G)$ is the group of self-homeomorphisms of the boundary of the tree (which is a Cantor set) that arise as ``mash-ups'' of $G$ with the Higman--Thompson group $V_d$. For $G$ the Grigorchuk group, $V_2(G)$ was introduced by R\"over in \cite{roever99} and is now called the \emph{R\"over group}. The groups $V_d(G)$ in general were introduced by Nekrashevych in \cite{nekrashevych04}. The self-similarity of $G$ is what ensures that $V_d(G)$ is a group.

The Houghton groups $H_n$ were introduced by Houghton in \cite{houghton79}, and much like the Thompson groups have served as interesting examples and counterexamples in group theory. Given a shift-similar group $G\le \Sym(\N)$, the Houghton-like group $H_n(G)$ we introduce here is the group of self-bijections of $\{1,\dots,n\}\times\N$ that arise as ``mash-ups'' of $G$ with the Houghton group $H_n$. The shift-similarity of $G$ is what ensures that $H_n(G)$ is a group. A new model we use here for describing elements of $H_n$ and $H_n(G)$ via so called representative triples (see Subsections~\ref{ssec:thomp_triples} and~\ref{ssec:houghton_triples}) has notable similarities to a standard model used to describe elements of Thompson groups and their relatives. While Thompson groups and Houghton groups are quite different in many ways, and these analogous models have key differences, we believe this highlights some connections between the world of Thompson and Houghton groups.

The first main result we prove about shift-similar groups is the following.

\begin{main:contain_symfin}
If $G\le\Sym(\N)$ is any infinite shift-similar group, then $\Symfin(\N)\le G$.
\end{main:contain_symfin}

Here $\Symfin(\N)$ denotes the normal subgroup of $\Sym(\N)$ consisting of all elements with finite support, that is, those $g$ such that $g(i)=i$ for all but finitely many $i\in\N$.

This result indicates that, from a group theoretic standpoint, shift-similarity is actually quite different from self-similarity. For example, all self-similar groups are residually finite, but Theorem~\ref{thrm:contain_symfin} shows that infinite shift-similar groups never are. (It turns out that the only finite shift-similar groups are the symmetric groups $\Sym(\{1,\dots,m\})$ -- see Observation~\ref{obs:finite_boring} -- so our focus is on infinite shift-similar groups.)

Our second main result is the following.

\begin{main:embed_fin_gen}
For any finitely generated group $\Gamma$, there exists a finitely generated shift-similar group $G\le\Sym(\N)$ such that $\Gamma$ embeds as a subgroup of $G$.
\end{main:embed_fin_gen}

In fact we prove that the shift-similar group $G$ in Theorem~\ref{thrm:embed_fin_gen} can always be taken to be \emph{strongly shift-similar}, which means that a certain injective endomorphism of $G/\Symfin(\N)$ is an automorphism (see Definition~\ref{def:strong}).

This result implies that there exist uncountably many isomorphism classes of finitely generated strongly shift-similar groups (Corollary~\ref{cor:uncountable}), so strongly shift-similar groups form an especially robust class of groups. In Example~\ref{ex:lots_of_fg_ss}, we spell out an intriguing family of examples of finitely generated shift-similar groups that warrant further investigation in the future. Theorem~\ref{thrm:embed_fin_gen} also implies that there exist finitely generated strongly shift-similar groups with undecidable word problem (Corollary~\ref{cor:undecidable}). Note that Theorem~\ref{thrm:embed_fin_gen} contrasts with the self-similar case, since no non-residually finite group can embed in any self-similar group, and moreover there only exist countably many isomorphism classes of finitely generated self-similar groups \cite[Subsection~1.5.3]{nekrashevych05}. Despite these group theoretic differences, shift-similar groups and self-similar groups are closely related in many ways, and their definitions are remarkably analogous.

Serving as a ``Houghton world'' analog of self-similar groups indicates to us that shift-similar groups are important, because self-similar groups are undoubtedly important. Most famously, Grigorchuk's self-similar group served as the first known example of a group with intermediate growth, and of an amenable group that is not elementary amenable \cite{grigorchuk80,grigorchuk84}. As another important example, self-similar groups were used in \cite{grigorchuk00} to disprove a strong version of the Atiyah conjecture. R\"over--Nekrashevych groups of self-similar groups have also proven to have important applications, for instance they were shown in \cite{skipper19} to provide the first known family of simple groups with arbitrary finiteness length.

The Houghton groups $H_n$ and generalizations have also received a lot of attention recently. Just to highlight a few examples, work has been done on their conjugacy problem \cite{antolin15}, twisted conjugacy problem \cite{cox17}, metric properties and automorphisms \cite{burillo16}, centralizers and Bredon cohomological properties \cite{stjohngreen15}, and Bieri--Neumann--Strebel--Renz invariants \cite{zaremsky17F_n,zaremsky20}. Braided versions of the $H_n$, introduced in \cite{degenhardt00}, have also received recent attention, e.g., in \cite{genevois22} and \cite{bux20}, as have higher dimensional analogs due to Bieri and Sach \cite{bieri22}.

One of the key properties of the Houghton groups is that they are elementary amenable. Another main result we prove here is that the amenability of $H_n(G)$ is tied to that of $G$:

\begin{main:amenable}
Let $G\le\Sym(\N)$ be shift-similar. Then $H_n(G)$ is amenable if and only if $G$ is amenable.
\end{main:amenable}

We actually prove that the same result holds with ``amenable'' replaced by either ``elementary amenable'' or ``contains no non-abelian free subgroups.'' Some other results we prove include that the $H_n(G)$ themselves are shift-similar (Proposition~\ref{prop:HnG_shift_sim}), that $H_n(G)$ is finitely generated as soon as $G$ is (Corollary~\ref{cor:fin_gen} and Observation~\ref{obs:H1_fin_gen}), and that if $G$ is strongly shift-similar then $H_1(G)=G$ (Proposition~\ref{prop:H1_equal}).

This paper is a starting point, and there are many questions and avenues for further inquiry into shift-similar and Houghton-like groups. Some concrete questions we pose here include Question~\ref{quest:not_strong} on whether there exists an infinite shift-similar group that is not strongly shift-similar, and Questions~\ref{quest:inherit_fin_props} and~\ref{quest:examples_fin_props} regarding the higher finiteness properties of the $H_n(G)$. One potential connection is with the ``locally defined'' groups described by Farley and Hughes in \cite{farley24}, which encompass both Houghton groups and R\"over--Nekrashevych groups. We think there is a chance that this framework applies to the $H_n(G)$, and could potentially be useful for inspecting their higher finiteness properties. It also seems like shift-similar and Houghton-like groups could fit into a ``Houghton world'' analog of cloning systems, as in \cite{witzel18,zaremsky18,witzel19}.

This paper is organized as follows. In Section~\ref{sec:thompson} we recall the background on Higman--Thompson groups $V_d$, self-similar groups, and R\"over--Nekrashevych groups $V_d(G)$. In Section~\ref{sec:houghton} we recall the background on Houghton groups $H_n$, and establish our Thompson-esque model for representing elements of $H_n$. In Section~\ref{sec:shift_similar} we introduce shift-similar groups, inspect their properties, and provide examples. In Section~\ref{sec:houghton_like} we introduce the Houghton-like groups $H_n(G)$, including the model for representing their elements, and prove that, thanks to $G$ being shift-similar, the $H_n(G)$ really are groups. Finally, in Section~\ref{sec:properties} we inspect a number of properties of the $H_n(G)$.

\subsection*{Acknowledgments} Thanks are due to Matt Brin and Anthony Genevois for helpful comments and discussions, and to the anonymous referee for several good suggestions. The second author is supported by grant \#635763 from the Simons Foundation.

\section{Preamble on Higman--Thompson groups, self-similar groups, and R\"over--Nekrashevych groups}\label{sec:thompson}

In this section we recall some background on Higman--Thompson groups, self-similar groups, and R\"over--Nekrashevych groups. We will do so in a way conducive to spelling out the analogy to Houghton groups, shift-similar groups, and our new Houghton-like groups in the coming sections.

\subsection{Higman--Thompson groups}\label{ssec:thomp}

Let $C_d=\{1,\dots,d\}^\N$ be the $d$-ary Cantor set, that is, the space of all infinite sequences $\kappa$ of elements of $\{1,\dots,d\}$ with the usual product topology. Let $\{1,\dots,d\}^*$ be the set of all finite strings $w$ of elements of $\{1,\dots,d\}$. Given $w\in \{1,\dots,d\}^*$, denote by $C_d(w)$ the \emph{cone} on $w$, defined via
\[
C_d(w)\coloneqq \{w\kappa\mid \kappa\in C_d\} \text{.}
\]
For any cone $C_d(w)$, there is a \emph{canonical homeomorphism} $h_w\colon C_d\to C_d(w)$ given by
\[
h_w \colon \kappa \mapsto w\kappa \text{.}
\]

\begin{definition}[Higman--Thompson group $V_d$]
The \emph{Higman--Thompson group} $V_d$ is the group of all homeomorphisms $f\colon C_d\to C_d$ defined by the following procedure:
\begin{enumerate}
    \item Take a partition of $C_d$ into finitely many cones $C_d(w_1^+),\dots,C_d(w_n^+)$.
    \item Take another partition of $C_d$ into the same number of cones $C_d(w_1^-),\dots,C_d(w_n^-)$.
    \item Map $C_d$ to itself bijectively by sending each $C_d(w_i^+)$ to some $C_d(w_j^-)$ via $h_{w_j^-} \circ h_{w_i^+}^{-1}$.
\end{enumerate}
\end{definition}

It turns out $V_d$ really is a group, i.e., a composition of homeomorphisms of this form is also of this form.

\subsection{Self-similar groups}\label{ssec:self_sim}

Now let us recall the definition of self-similar groups. See \cite{nekrashevych05} for a wealth of background. We will define them in a slightly non-standard way that will demonstrate the connection to shift-similar groups later. Let $\tree_d$ be the infinite rooted $d$-ary tree, so the vertex set of $\tree_d$ is $\{1,\dots,d\}^*$, the empty word $\varnothing$ is the root, and two vertices are adjacent if they are of the form $w$, $wi$ for some $i\in\{1,\dots,d\}$. For each $1\le i\le d$, let $\tree_d(i)$ be the induced subgraph of $\tree_d$ spanned by all vertices of the form $iw$ for $w\in\{1,\dots,d\}^*$. Note that $\tree_d(i)$ is naturally isomorphic to $\tree_d$, via the graph isomorphism
\[
\delta_i\colon \tree_d \to \tree_d(i)
\]
sending the vertex $w$ to $iw$.

Let $\Aut(\tree_d)$ be the group of automorphisms of $\tree_d$. Since the root is the only vertex of degree $d$ (the others have degree $d+1$), every automorphism fixes the root. Thus, every automorphism stabilizes the ``level-$1$ set'' of vertices $\{1,\dots,d\}$. In particular we get an epimorphism $\rho\colon \Aut(\tree_d)\to S_d$. Note that the kernel of $\rho$ is the subgroup of automorphisms fixing every level-$1$ vertex, which is isomorphic to $\Aut(\tree_d)^d$, and the map $\rho$ clearly splits, so we have
\[
\Aut(\tree_d) \cong S_d \ltimes \Aut(\tree_d)^d \text{,}
\]
or, more concisely, $\Aut(\tree_d) \cong S_d \wr \Aut(\tree_d)$. (It turns out to be convenient to write semidirect and wreath products with the acting group on the left.)

Now for each $1\le i\le d$ let
\[
\state_i \colon \Aut(\tree_d)\to\Aut(\tree_d)
\]
be the function (not homomorphism)
\[
\state_i(g)\coloneqq \delta_{\rho(g)(i)}^{-1}\circ g|_{\tree_d(i)}\circ \delta_i \text{.}
\]
Note that the image of $g|_{\tree_d(i)}\circ \delta_i$ is $\tree_d(\rho(g)(i))$, so the composition $\delta_{\rho(g)(i)}^{-1}\circ g|_{\tree_d(i)}\circ \delta_i$ makes sense and is again an element of $\Aut(\tree_d)$. To tie this to the above, the isomorphism $\Aut(\tree_d) \cong S_d \ltimes \Aut(\tree_d)^d$ can now be written as $g\mapsto (\rho(g),(\state_1(g),\dots,\state_d(g)))$.

\begin{definition}[Self-similar]\label{def:self_similar}
Let $G\le \Aut(\tree_d)$. We call $G$ \emph{self-similar} if for all $g\in G$ and all $1\le i\le d$ we have $\state_i(g)\in G$.
\end{definition}

\begin{example}[Grigorchuk group]\label{ex:grig}
Perhaps the most prominent example of a self-similar group is the Grigorchuk group, introduced by Grigorchuk in \cite{grigorchuk80}. This is the subgroup of $\Aut(\tree_2)$ generated by elements $a$, $b$, $c$, and $d$, described as follows. Identifying $\Aut(\tree_2)$ with $S_2\ltimes \Aut(\tree_2)^2$ via the above isomorphism, the elements $a$, $b$, $c$, and $d$ are defined recursively by $a=(1~2)(\id,\id)$, $b=(a,c)$, $c=(a,d)$, and $d=(\id,b)$. It turns out the Grigorchuk group has intermediate growth, and is amenable but not elementary amenable, serving as the first known example of such a group \cite{grigorchuk84}.
\end{example}

It turns out that there exist only countably many (isomorphism classes of) finitely generated self-similar groups, see, e.g., \cite[Subsection~1.5.3]{nekrashevych05}. This will contrast with the shift-similar situation later (see Corollary~\ref{cor:uncountable}).

\subsection{R\"over--Nekrashevych groups}\label{ssec:ro_nekr}

In this subsection we recall the definition of the R\"over--Nekrashevych group $V_d(G)$ of a self-similar group $G$. The first example, now called the \emph{R\"over group}, was introduced by R\"over in \cite{roever99} using the Grigorchuk group as $G$. The groups were introduced in full generality by Nekrashevych in \cite{nekrashevych04}.

\begin{definition}[R\"over--Nekrashevych group]\label{def:nekr}
Let $G\le \Aut(\tree_d)$ be self-similar. The \emph{R\"over--Nekrashevych group} $V_d(G)$ is the group of all homeomorphisms $f\colon C_d\to C_d$ defined by the following procedure:
\begin{enumerate}
    \item Take a partition of $C_d$ into finitely many cones $C_d(w_1^+),\dots,C_d(w_n^+)$.
    \item Take another partition of $C_d$ into the same number of cones $C_d(w_1^-),\dots,C_d(w_n^-)$.
    \item Map $C_d$ to itself bijectively by sending each $C_d(w_i^+)$ to some $C_d(w_j^-)$ via $h_{w_j^-} \circ g_i\circ h_{w_i^+}^{-1}$ for some $g_1,\dots,g_n\in G$.
\end{enumerate}
\end{definition}

Note that the only difference between this and the definition of the Higman--Thompson group $V_d$ is the presence of the elements $g_i$ in the last step. Intuitively, in $V_d$ we map one cone to another cone by identifying them both with $C_d$ via canonical homeomorphisms and then mapping $C_d$ to itself via the identity, whereas in $V_d(G)$ we do the same thing but map $C_d$ to itself via an element of $G$. Note that $C_d=\partial\tree_d$, so $\Aut(\tree_d)$ acts naturally on $C_d$, and this is the sense in which we are viewing elements of $G$ as homeomorphisms of $C_d$.

It turns out $V_d(G)$ really is a group, thanks to $G$ being self-similar. This is basically because, given two elements, up to refining partitions we can assume that the range partition of the first element is the same as the domain partition of the second element, and then composing the two elements is straightforward. Self-similarity of $G$ ensures that, when we refine the domain and range partitions of a given element, the resulting $g_i$ are still in $G$. This will also be made more precise in the next subsection.

\subsection{Representative triples}\label{ssec:thomp_triples}

In this subsection we discuss a convenient way of representing an element of $V_d$ or $V_d(G)$, via so called representative triples. This is the viewpoint taken for example in \cite{skipper21}. First note that if $C(w_1),\dots,C(w_n)$ is a partition of $C_d$ into cones, then $w_1,\dots,w_n$ are the leaves of a finite subtree of $\tree_d$. Thus, the step, ``take a partition of $C_d$ into $n$ cones,'' is equivalent to, ``take a finite subtree of $\tree_d$ with $n$ leaves.'' Here by a \emph{subtree} of $\tree_d$ we always mean one containing the root, and such that for any vertex $v$ in the subtree, if $vi$ is in the subtree for some $1\le i\le d$, then so is $vj$ for all $1\le j\le d$. Let us just say \emph{finite $d$-ary tree} for such a subtree from now on.

Now an element of $V_d(G)$ can be represented by a triple $(T_-,\sigma(g_1,\dots,g_n),T_+)$, where $T_+$ and $T_-$ are finite $d$-ary trees with $n$ leaves, $\sigma$ is an element of $S_n$, and the $g_i$ are elements of $G$. The element this triple represents is the one obtained by the procedure from Definition~\ref{def:nekr}, with $w_1^+,\dots,w_n^+$ the leaves of $T_+$, $w_1^-,\dots,w_n^-$ the leaves of $T_-$, and $\sigma$ the permutation sending $i$ to $j$ whenever $C(w_i^+)$ is being mapped to $C(w_j^-)$.

Call any $(T_-,\sigma(g_1,\dots,g_n),T_+)$ as above a \emph{representative triple}. An element of $V_d(G)$ can be represented by more than one representative triple, dictated by a certain equivalence relation. The key is the notion of an ``expansion'' of a representative triple.

\begin{definition}[Expansion]
Let $(T_-,\sigma(g_1,\dots,g_n),T_+)$ be a representative triple for an element of $V_d(G)$. Let $1\le k\le n$. The $k$th \emph{expansion} of this triple is the triple
\[
(T_-',\sigma'(g_1,\dots,g_{k-1},\state_1(g_k),\dots,\state_d(g_k),g_{k+1},\dots,g_n),T_+')
\]
defined as follows. The tree $T_+'$ is obtained from $T_+$ by adding a $d$-ary caret to the $k$th leaf, i.e., replacing the leaf $w$ with the leaves $w1,\dots,wd$. The tree $T_-'$ is obtained from $T_-$ by adding a $d$-ary caret to the $\sigma(k)$th leaf. The permutation $\sigma'$ is given by, roughly, replacing $k$ in the domain of $\sigma$ with a $d$-element set, replacing $\sigma(k)$ in the range of $\sigma$ with another copy of this $d$-element set, taking the former to the latter via $\rho(g_k)$, and acting on $\{1,\dots,k-1,k+1,\dots,n\}$ like $\sigma$. See \cite[Figure~1]{skipper21} for more precision about $\sigma'$.
\end{definition}

If one triple is an expansion of the other, call the second a \emph{reduction} of the first. Two triples are \emph{equivalent} if they can be obtained from each other by a finite sequence of expansions and reductions. It turns out two triples represent the same element of $V_d(G)$ if and only if they are equivalent, so we can view $V_d(G)$ as the set of equivalence classes, denoted $[T_-,\sigma(g_1,\dots,g_n),T_+]$, of representative triples $(T_-,\sigma(g_1,\dots,g_n),T_+)$.

The group operation on $V_d(G)$ is easy to realize using equivalence classes of representative triples. Given two elements of $V_d(G)$, say $[T_-,\sigma(g_1,\dots,g_n),T_+]$ and $[U_-,\tau(h_1,\dots,h_m),U_+]$, we can perform expansions until without loss of generality $T_+=U_-$ (so $m=n$). Then the group operation is given by
\[
[T_-,\sigma(g_1,\dots,g_n),T_+][T_+,\tau(h_1,\dots,h_n),U_+] = [T_-,\sigma(g_1,\dots,g_n)\tau(h_1,\dots,h_n),U_+]\text{.}
\]
Note that $\sigma(g_1,\dots,g_n)\tau(h_1,\dots,h_n)$ is a product of elements of $S_n \ltimes G^n$, and can be rewritten in the standard form we use in representative triples if desired, namely as $(\sigma\circ\tau)(g_{\tau(1)}h_1,\dots,g_{\tau(n)}h_n)$.

This concludes our preamble on the relevant aspects of the world of Thompson groups, setting the stage for our work establishing parallels to the world of Houghton groups. In the coming sections, we will discuss the Houghton groups in a way that makes clear the analogy to Higman--Thompson groups, then introduce shift-similar groups as a ``Houghton world'' analog of self-similar groups, and finally introduce our Houghton-like groups as an analog of R\"over--Nekrashevych groups.

\section{Houghton groups}\label{sec:houghton}

Let $n\in\N$ and let $[n]\coloneqq \{1,\dots,n\}$. The Houghton group $H_n$, introduced in \cite{houghton79}, is the group of self-bijections of $[n]\times\N$ that are \emph{eventual translations}, that is, translations outside a finite subset. More precisely, a bijection $\eta\colon [n]\times\N \to [n]\times\N$ is in $H_n$ if there exist $m_1,\dots,m_n\in\Z$ and there exists a finite subset $M\subseteq [n]\times\N$ such that for all $(k,i)\in ([n]\times\N)\setminus M$ we have $\eta(k,i)=(k,i+m_k)$.

Let us set up an equivalent description of $H_n$, which is reminiscent of the description in Section~\ref{sec:thompson} of the Higman--Thompson groups $V_d$. First, for $1\le k\le n$, and $M$ a finite subset of $[n]\times\N$, define the \emph{quasi-ray} $Q(k,M)$ to be
\[
Q(k,M)\coloneqq (\{k\}\times \N)\setminus M \text{.}
\]
The name comes from viewing each $\{k\}\times\N$ as a ``ray,'' so a quasi-ray is the result of removing finitely many points from a ray. Each ray admits a canonical bijection with $\N$, given by enumerating $Q(k,M)$ as $Q(k,M)=\{(k,i_1),(k,i_2),\dots\}$, with $i_1<i_2<\cdots$. Let us denote the \emph{canonical bijection} corresponding to $k$ and $M$ by $\beta_{k,M}\colon \N\to Q(k,M)$, defined via
\[
\beta_{k,M}(j)\coloneqq (k,i_j)\text{,}
\]
with $i_j$ as above. Note that since $M$ is finite, for sufficiently large $j$ we have that if $\beta_{k,M}(j)=(k,i)$, then $\beta_{k,M}(j+1)=(k,i+1)$.

\begin{definition}[Houghton groups]\label{def:houghton}
The Houghton group $H_n$ is the group of bijections from $[n]\times\N$ to itself given as follows:
\begin{enumerate}
    \item Take a partition of $[n]\times\N$ into a finite subset $M_+$ and the corresponding quasi-rays $Q(1,M_+),\dots,Q(n,M_+)$.
    \item Take another partition of $[n]\times\N$ into a finite subset $M_-$ with $|M_-|=|M_+|$ and the corresponding quasi-rays $Q(1,M_-),\dots,Q(n,M_-)$.
    \item Map $[n]\times\N$ to itself by sending $M_+$ to $M_-$ via some bijection $\sigma$ and for each $1\le k\le n$ sending $Q(k,M_+)$ to $Q(k,M_-)$ via $\beta_{k,M_-}\circ \beta_{k,M_+}^{-1}$.
\end{enumerate}
\end{definition}

To see why this description recovers the previous one, first note that for any finite subset $M$ of $[n]\times\N$ and any quasi-ray $Q(k,M)$, there exists $m_k\in\Z$ such that for $i\in\N$ sufficiently large we have $\beta_{k,M}(i)=(k,i+m_k)$. Thus, each $\beta_{k,M_-}\circ \beta_{k,M_+}^{-1}$ as above is an eventual translation.

\subsection{Representative triples}\label{ssec:houghton_triples}

Just like with the Higman--Thompson groups, as we now show, elements of the Houghton groups can be represented by certain triples. First note that any partition of $[n]\times\N$ into a finite subset and $n$ quasi-rays is completely determined by the finite subset. Moreover, any bijection constructed as in Definition~\ref{def:houghton} is determined by the finite subsets $M_+$ and $M_-$, and the bijection $\sigma\colon M_+\to M_-$. Hence, any element of $H_n$ is completely determined by the \emph{representative triple} $(M_-,\sigma,M_+)$.

We should caution that there is a fundamental difference between this model for elements of Houghton groups and the one for Higman--Thompson groups. In both cases, a finite ``chunk'' is singled out in both the domain and the range, but in this Houghton model, $\sigma$ serves to map things in the finite chunk of the domain to things in the finite chunk of the range, whereas in the Higman--Thompson model $\sigma$ serves to map things outside the finite chunk of the domain to things outside the finite chunk of the range. This discrepancy will be further commented on in Remark~\ref{rmk:quasi_auts}. Thus, all the analogies we are discussing here between the Thompson and Houghton world should not be taken as a literal comparison of the groups themselves so much as a comparison between techniques used to analyze the groups.

An element of $H_n$ can be represented by more than one representative triple. For example, the identity is represented by any $(M,\id_M,M)$. The following notion of expansion will turn out to completely encode which triples represent the same element of $H_n$.

\begin{definition}[Expansion/reduction/equivalent]
Let $(M_-,\sigma,M_+)$ be a representative triple. The $k$th \emph{expansion} of the triple $(M_-,\sigma,M_+)$ is the triple
\[
(M_-\cup\{\beta_{k,M_-}(1)\},\sigma',M_+\cup\{\beta_{k,M_+}(1)\})\text{,}
\]
where $\sigma'$ is the bijection from $M_+\cup\{\beta_{k,M_+}(1)\}$ to $M_-\cup\{\beta_{k,M_-}(1)\}$ sending $\beta_{k,M_+}(1)$ to $\beta_{k,M_-}(1)$ and otherwise acting like $\sigma$. If one triple is an expansion of another, call the second a \emph{reduction} of the first. Call two representative triples \emph{equivalent} if one can be obtained from the other by a finite sequence of expansions and reductions. Write $[M_-,\sigma,M_+]$ for the equivalence class of $(M_-,\sigma,M_+)$.
\end{definition}

The intuition behind expansion is, we enlarge the finite set $M_+$ by adding in the ``first'' element of the domain quasi-ray $Q(k,M_+)$, and similarly add in the first element of the range quasi-ray $Q(k,M_-)$, and extend $\sigma$ to $\sigma'$ appropriately. Thus we change the partitions, but not the overall self-bijection of $[n]\times\N$.

There is a certain kind of move within an equivalence class of representative triples that will be useful to use later, and so is convenient to give a name to. We call it a general expansion, and it amounts to doing the same move as an expansion but using $\beta_{k,M_+}(j)$ and $\beta_{k,M_-}(j)$ for an arbitrary $j$ instead of just $j=1$.

\begin{definition}[General expansion/reduction]
Let $(M_-,\sigma,M_+)$ be a representative triple. The \emph{general expansion} of this triple associated to the pair $(k,j)$ is the triple
\[
(M_-\cup\{\beta_{k,M_-}(j)\},\sigma',M_+\cup\{\beta_{k,M_+}(j)\})\text{,}
\]
where $\sigma'$ is the bijection from $M_+\cup\{\beta_{k,M_+}(j)\}$ to $M_-\cup\{\beta_{k,M_-}(j)\}$ sending $\beta_{k,M_+}(j)$ to $\beta_{k,M_-}(j)$ and otherwise acting like $\sigma$. The reverse operation is called a \emph{general reduction}.
\end{definition}

\begin{lemma}\label{lem:gen_exp}
If one representative triple is a general expansion of the other, then they are equivalent.
\end{lemma}

\begin{proof}
Say the first triple is $(M_-,\sigma,M_+)$ and the second is
\[
(M_-\cup\{\beta_{k,M_-}(j)\},\sigma',M_+\cup\{\beta_{k,M_+}(j)\})
\]
as above. Let us induct on $j$. If $j=1$ then this is just an expansion, and we are done. Now assume $j>1$. By induction $(M_-,\sigma,M_+)$ is equivalent via a general expansion to
\[
(M_-\cup\{\beta_{k,M_-}(j-1)\},\sigma'',M_+\cup\{\beta_{k,M_+}(j-1)\})\text{,}
\]
for appropriate $\sigma''$. Now observe that $\beta_{k,M_+\cup\{\beta_{k,M_+}(j-1)\}}(j-1)$ equals $\beta_{k,M_+}(j)$ (with a similar statement for the $M_-$ version), so again by induction our triple is equivalent via a general expansion to
\[
(M_-\cup\{\beta_{k,M_-}(j-1),\beta_{k,M_-}(j)\},\sigma''',M_+\cup\{\beta_{k,M_+}(j-1),\beta_{k,M_+}(j)\})\text{,}
\]
for appropriate $\sigma'''$. Next we note that $\beta_{k,M_+\cup\{\beta_{k,M_+}(j)\}}(j-1)$ equals $\beta_{k,M_+}(j-1)$ (with a similar statement for the $M_-$ version), so again by induction (now doing a general reduction) our triple is equivalent to
\[
(M_-\cup\{\beta_{k,M_-}(j)\},\sigma',M_+\cup\{\beta_{k,M_+}(j)\})\text{,}
\]
so we are done.
\end{proof}

Note that given any two representative triples $(M_-,\sigma,M_+)$ and $(N_-,\tau,N_+)$, we can perform a sequence of general expansions to each, say to get $(M_-',\sigma',M_+')$ and $(N_-',\tau',N_+')$ respectively, such that $M_+'=N_+'=M_+\cup N_+$. We can also perform some (other) sequence of general expansions until $M_+'=N_-'=M_+\cup N_-$, or $M_-'=N_+'=M_-\cup N_+$, or $M_-'=N_-'=M_-\cup N_-$. In particular, up to equivalence we can assume that one of the finite sets in the first triple equals one of the finite sets in the second triple, and we can pick which ones we want to be equal.

\begin{lemma}\label{lem:triples}
Two representative triples are equivalent if and only if the elements of $H_n$ they represent are equal.
\end{lemma}

\begin{proof}
Let $\eta,\theta\in H_n$ have equivalent representative triples, and we want to show $\eta=\theta$. Without loss of generality the triples differ by a single expansion, say there is a representative triple of $\theta$ that is an expansion of one for $\eta$. Say $\eta$ has $(M_-,\sigma,M_+)$ as a representative triple and $\theta$ has $(M_-',\sigma',M_+')$, where $M_+'=M_+\cup\{\beta_{k,M_+}(1)\}$ and $M_-'=M_-\cup\{\beta_{k,M_-}(1)\}$. Now we want to show that $\theta(\ell,i)=\eta(\ell,i)$ for all $(\ell,i)\in [n]\times\N$. First suppose $(\ell,i)\in M_+$. Then since $\sigma'(\ell,i)=\sigma(\ell,i)$ we indeed have $\theta(\ell,i)=\eta(\ell,i)$. Next suppose $(\ell,i)=\beta_{k,M_+}(1)$ (so $\ell=k$). Then
\[
\theta(k,i)=\beta_{k,M_-}(1)=\beta_{k,M_-}\circ \beta_{k,M_+}^{-1}(k,i)=\eta(k,i)
\]
as desired. Finally suppose $(\ell,i)\in Q(\ell,M_+')$. If $\ell\ne k$ then $Q(\ell,M_+')=Q(\ell,M_+)$ and $Q(\ell,M_-')=Q(\ell,M_-)$, so
\[
\theta(\ell,i)=\beta_{\ell,M_-'}\circ \beta_{\ell,M_+'}^{-1}(\ell,i)=\beta_{\ell,M_-}\circ \beta_{\ell,M_+}^{-1}(\ell,i)=\eta(\ell,i) \text{.}
\]
Now suppose $\ell=k$, and say $(k,i)=\beta_{k,M_+'}(j)$. Note that $\beta_{k,M_+'}(j)=\beta_{k,M_+}(j+1)$ and $\beta_{k,M_-'}(j)=\beta_{k,M_-}(j+1)$. Thus
\begin{align*}
\theta(k,i)&=\beta_{k,M_-'}\circ \beta_{k,M_+'}^{-1}(k,i)\\
&=\beta_{k,M_-'}(j)\\
&=\beta_{k,M_-}(j+1)\\
&=\beta_{k,M_-} \circ \beta_{k,M_+}^{-1}\circ \beta_{k,M_+'}(j)\\
&=\beta_{k,M_-} \circ \beta_{k,M_+}^{-1}(k,i)\\
&= \eta(k,i) \text{,}
\end{align*}
and we are done.

Now let $(M_-,\sigma,M_+)$ and $(N_-,\tau,N_+)$ be representative triples for the same element $\eta$, and we claim they are equivalent. It suffices to show that they have a common general expansion. Note that both $(M_-,\sigma,M_+)$ and $(N_-,\tau,N_+)$ have general expansions of the form $(M_-',\sigma',M_+\cup N_+)$ and $(N_-',\tau',M_+\cup N_+)$ respectively, for some $M_-'$, $\sigma'$, $N_-'$, and $\tau'$. Since these both represent $\eta$ (by the first paragraph), we see that $M_-'$ and $N_-'$ both equal $\eta(M_+\cup N_+)$, hence are equal themselves. Also, $\sigma'$ and $\tau'$ each equal the restriction of $\eta$ to $M_+\cup N_+$, hence are equal. We conclude that $(M_-',\sigma',M_+\cup N_+)=(N_-',\tau',M_+\cup N_+)$ is a common general expansion of the original triples, as desired.
\end{proof}

We can now discuss the group structure of $H_n$ using equivalence classes of representative triples. Given $[M_-,\sigma,M_+]$ and $[N_-,\tau,N_+]$, the product $[M_-,\sigma,M_+][N_-,\tau,N_+]$ is given by first performing general expansions until without loss of generality $M_+=N_-$, and then declaring that
\[
[M_-,\sigma,M_+][M_+,\tau,N_+]\coloneqq [M_-,\sigma\circ\tau,N_+] \text{.}
\]
This clearly models the usual multiplication operation in $H_n$, so now we can view $H_n$ as the group of equivalence classes of representative triples.

\begin{remark}[Quasi-automorphisms]\label{rmk:quasi_auts}
As we have emphasized, this model for elements of $H_n$ differs from the one for $V_d$ in that, here we keep track of a bijection shuffling the finite ``chunks'' that have been singled out, whereas in $V_d$ we keep track of a way of shuffling the pieces outside the finite chunks (trees) that have been singled out. There is a Thompson-like group where we additionally keep track of a shuffling of the finite chunks, i.e., the vertices of the finite trees. This is the group $QV_d$ of \emph{quasi-automorphisms} of $\tree_d$. See \cite{lehnert08,nucinkis18} for more background, and see \cite[Proposition~3.2]{audino18} for a discussion of what we could call representative quadruples for elements of these groups (at least for $d=2$). It seems like one could also merge $QV_d$ with self-similar groups $G$ to get groups we could call quasi-R\"over--Nekrashevych groups $QV_d(G)$, which would also have elements describable using a similar model to the ones we are using here. We will leave any further analysis of $QV_d$ or $QV_d(G)$ (the latter of which has not been formally defined anyway) for future work.
\end{remark}

\section{Shift-similar groups}\label{sec:shift_similar}

In this section we introduce shift-similar groups, as a ``Houghton world'' analog of the self-similar groups from Subsection~\ref{ssec:self_sim}. As before, $\Sym(\N)$ denotes the group of self-bijections of $\N$.

\subsection{The definition and some first examples}\label{ssec:def_first_ex}

\begin{definition}[Shift-similar, shifting maps]\label{def:shift_similar}
For each $j\in\N$, let
\[
s_j\colon \N\to\N\setminus\{j\}
\]
be the bijection sending $i$ to $i$ for all $1\le i<j$ and $i$ to $i+1$ for all $i\ge j$. For each $j\in\N$, let
\[
\shift_j\colon \Sym(\N)\to\Sym(\N)
\]
be the function (not homomorphism)
\[
\shift_j(g)\coloneqq s_{g(j)}^{-1}\circ g|_{\N\setminus\{j\}}\circ s_j \text{.}
\]
Note that the image of $g|_{\N\setminus\{j\}}\circ s_j$ is $\N\setminus\{g(j)\}$, so the composition $s_{g(j)}^{-1}\circ g|_{\N\setminus\{j\}}\circ s_j$ makes sense and is again an element of $\Sym(\N)$. Let $G$ be a subgroup of $\Sym(\N)$. Call $G$ \emph{shift-similar} if for all $g\in G$ and all $j\in\N$ we have $\shift_j(g)\in G$. Call $\shift_j$ the $j$th \emph{shifting map}.
\end{definition}

Intuitively, if we view $g\in G$ via a picture of the domain copy of $\N$, the range copy of $\N$, and an arrow from $i$ in the domain to $g(i)$ in the range, for each $i$, then $\shift_j(g)$ is obtained by deleting $j$ from the domain, deleting $g(j)$ from the range, erasing the arrow from $j$ to $g(j)$, and renumbering the resulting sets $\N\setminus\{j\}$ and $\N\setminus\{g(j)\}$ to again be $\N$. See Figure~\ref{fig:shift} for an example.

\begin{figure}[htb]\centering
 \begin{tikzpicture}[line width=1pt]
    \draw[dashed,->] (1,0) -- (3,1);
    \draw[->] (2,0) -- (5,1);
    \draw[->] (3,0) -- (1,1);
    \draw[->] (4,0) -- (2,1);
    \draw[->] (5,0) -- (6,1);
    \draw[->] (6,0) -- (4,1);
    
    \node at (1,-.3) {$1$}; 
    \node at (2,-.3) {$2$};
    \node at (3,-.3) {$3$};
    \node at (4,-.3) {$4$};
    \node at (5,-.3) {$5$};
    \node at (6,-.3) {$6$};
    \node at (1,1.3) {$1$};
    \node at (2,1.3) {$2$};
    \node at (3,1.3) {$3$};
    \node at (4,1.3) {$4$};
    \node at (5,1.3) {$5$};
    \node at (6,1.3) {$6$};
    
    \node at (6.5,0.5) {$\cdots$};
    
    \node at (8,0.5) {$\longrightarrow$};
    
   \begin{scope}[xshift=8cm]
    \draw[->] (1,0) -- (4,1);
    \draw[->] (2,0) -- (1,1);
    \draw[->] (3,0) -- (2,1);
    \draw[->] (4,0) -- (5,1);
    \draw[->] (5,0) -- (3,1);
    
    \node at (1,-.3) {$1$}; 
    \node at (2,-.3) {$2$};
    \node at (3,-.3) {$3$};
    \node at (4,-.3) {$4$};
    \node at (5,-.3) {$5$};
    \node at (1,1.3) {$1$};
    \node at (2,1.3) {$2$};
    \node at (3,1.3) {$3$};
    \node at (4,1.3) {$4$};
    \node at (5,1.3) {$5$};
    
    \node at (5.5,0.5) {$\cdots$};
   \end{scope}
 \end{tikzpicture}
\caption{An example of an element $g\in\Sym(\N)$ and its image under $\shift_1$. Here $g$ restricted to $\{1,\dots,6\}$ is the permutation $(1~3)(2~5~6~4)$, and $\shift_1(g)$ restricted to $\{1,\dots,5\}$ turns out to be the permutation $(1~4~5~3~2)$ (the restriction of $g$ to $\{7,8,\dots\}$ is left up to the reader's imagination). The dashed line indicates the arrow that gets deleted when computing $\shift_1$, namely the one originating at the domain copy of $1$.}\label{fig:shift}
\end{figure}

It is also useful to have the following in mind, which is a ``casewise'' definition for $\shift_j(g)$.
\[
\psi_j(g)(i) = \left\{\begin{array}{ll}
g(i) & \text{if } i<j \text{ and } g(i)<g(j) \\
g(i)-1 & \text{if } i<j \text{ and } g(i)>g(j) \\
g(i+1) & \text{if } i\ge j \text{ and } g(i+1)<g(j) \\
g(i+1)-1 & \text{if } i\ge j \text{ and } g(i+1)\ge g(j) \text{.}
\end{array}\right.
\]
Also note that, iterating shifting maps, shift-similarity asks that upon removing any finite subset of $\N$ and applying $g$ to the result, the induced element of $\Sym(\N)$ (after identifying the complement of the finite subsets of the domain and range with $\N$ via canonical bijections) again lies in $G$.

\begin{remark}
Given a group $G$, we will often refer to $G$ as \emph{shift-similar} if it admits an embedding into $\Sym(\N)$ such that the image (which is isomorphic to $G$) is shift-similar. We may call this a \emph{shift-similar representation} of the group. Thus we can ask, ``is this group shift-similar?'' even if it is not handed to us as a subgroup of $\Sym(\N)$. Note that a shift-similar group may admit an embedding into $\Sym(\N)$ that is not shift-similar, but we still call it shift-similar if it admits at least one embedding whose image is. (All of these sorts of terminological caveats also hold for self-similar groups of automorphisms of trees.)
\end{remark}

Let us discuss some initial properties of shifting maps. This first result actually shows that the $\shift_j$ exhibit some sort of ``Thompson-like'' behavior. (We will not investigate this Thompson-like behavior further here; in particular the fact that the $\shift_j$ are not injective makes it unclear what the actual connection to Thompson's groups might be.)

\begin{lemma}\label{lem:two_shifts}
For any $j\le j'$ we have $\shift_{j'} \circ \shift_j = \shift_j \circ \shift_{j'+1}$.
\end{lemma}

\begin{proof}
Rather than wade through a long list of cases, we will appeal to the intuitive picture. For a given $g\in G$, if we delete $j$ from the domain, delete $g(j)$ from the range, delete the arrow from $j$ to $g(j)$, and renumber appropriately, then what used to be the arrow from $j'+1$ to $g(j'+1)$ becomes an arrow from $j'$ to some range point (either $g(j'+1)-1$ or $g(j'+1)$ depending on where $g(j'+1)$ was in relation to $g(j)$). Hence, deleting the arrow originating at $j'$ after having already deleted the arrow originating at $j$ is equivalent to first deleting the arrow originating at $j'+1$ and then deleting the arrow originating at $j$.
\end{proof}

We also have that, while each shifting map $\shift_j$ is not a homomorphism, the family of all shifting maps does satisfy the following nice property reminiscent of homomorphisms (or really cocycles):

\begin{observation}\label{obs:almost_hom}
For any $j\in\N$ and any $g,h\in \Sym(\N)$, we have
\[
\shift_j(g\circ h)=\shift_{h(j)}(g)\circ \shift_j(h) \text{.}
\]
\end{observation}

\begin{proof}
By definition we have
\begin{align*}
\shift_j(g\circ h)&= s_{g(h(j))}^{-1}\circ (g\circ h)|_{\N\setminus\{j\}}\circ s_j\\
&= s_{g(h(j))}^{-1}\circ g|_{\N\setminus\{h(j)\}}\circ s_{h(j)} \circ s_{h(j)}^{-1} \circ h|_{\N\setminus\{j\}}\circ s_j\\
&= \shift_{h(j)}(g)\circ \shift_j(h) \text{.}
\end{align*}
\end{proof}

This also lets us pin down how shifting maps interact with inverses.

\begin{corollary}\label{cor:inverse}
For any $j\in\N$ and any $g\in \Sym(\N)$, we have
\[
\shift_j(g)^{-1} = \shift_{g(j)}(g^{-1}) \text{.}
\]
\end{corollary}

\begin{proof}
Clearly $\shift_j(\id)=\id$ for all $j$. By Observation~\ref{obs:almost_hom}, we get
\[
\id = \shift_j(g^{-1}\circ g) = \shift_{g(j)}(g^{-1}) \circ \shift_j(g) \text{,}
\]
so indeed $\shift_j(g)^{-1} = \shift_{g(j)}(g^{-1})$.
\end{proof}

Now we discuss some immediate examples of shift-similar groups. The trivial group is the most obvious example, and more generally $S_m=\Sym(\{1,\dots,m\})$ for any $m\in\N$, but we will be more interested in infinite shift-similar groups. In fact the $S_m$ are the only finite shift-similar groups, as we now show.

\begin{lemma}\label{lem:finite_shift_sim}
Let $G\le \Sym(\N)$ be a finite shift-similar group. Then $G=S_m$ for some $m$.
\end{lemma}

\begin{proof}
Let $m$ be the maximum element of the orbit $G.1$, which exists since $G$ is finite. First we claim that $G\le S_m$. If $G$ is not contained in $S_m$, then there exist $i,j\in\N$ and $g\in G$ such that $m\le i<j$ and $g(i)=j$. Since $1\le i$ and $g(1)\le j$, we have $\shift_1(g)(i-1)=j-1$, and doing this $i-m$ times we see that without loss of generality $i=m$. But this means $G.1$ contains $j$, contradicting the maximality of $m$. We conclude $G\le S_m$. Now we need to show that $G$ is all of $S_m$. We induct on $m$; the base case $m=1$ holds trivially, so let $m>1$. Choose $g\in G$ such that $g(1)=m$. This implies that $\shift_m(g)(1)=m-1$, since $1<m$ and $g(1)>g(m)$. Set $H\coloneqq G\cap S_{m-1}$, so $H$ is shift-similar. We have $\shift_m(g)\in H$, so $m-1$ is the maximum element of $H.1$. By induction, $H=S_{m-1}$. Now for any $g'\in S_m\setminus S_{m-1}$, say with $g'(m)=i<m$, we have $g\circ (1~i)\circ g' \in S_{m-1}\le G$. Since $g,(1~i)\in G$, so we conclude $g'\in G$, and so $S_m\le G$ as desired.
\end{proof}

Having completely classified finite shift-similar groups, we turn our attention to infinite shift-similar groups. We will discuss more complicated examples later, but for now let us point out a few obvious (but important) ones.

\begin{example}[Finitely supported permutations]
The most obvious example of an infinite shift-similar group is $\Symfin(\N)$. Indeed, $\Symfin(\N)$ is the directed union of the $S_m$, and it is clear that a directed union of shift-similar groups is shift-similar.
\end{example}

\begin{example}[Eventually periodic permutations]\label{ex:eventually_periodic}
For $n\in\N$, let $E_n\le \Sym(\N)$ be the group of all eventually $n$-periodic permutations, that is all $g$ satisfying $g(i+n)=g(i)+n$ for all sufficiently large $i$. (Note that $E_1$ is exactly $\Symfin(\N)$.) It is an easy exercise to check that $E_n$ is shift-similar.
\end{example}

In fact, $E_n$ is closely related to the Houghton group $H_n$, as the following shows.

\begin{lemma}\label{lem:even_per_houghton}
The group $E_n$ is isomorphic to $S_n \ltimes H_n$.
\end{lemma}

\begin{proof}
Consider the bijection $\xi\colon [n]\times\N \to \N$ sending $(k,i)$ to $k+(i-1)n$. Conjugation by $\xi$ gives an isomorphism $\omega\colon \Sym(\N)\to \Sym([n]\times\N)$, via $\omega(g)=\xi^{-1}\circ g\circ \xi$. We claim that the image of $E_n$ under $\omega$ is the natural copy of $S_n\ltimes H_n$ in $\Sym([n]\times\N)$, where $S_n$ acts on $[n]$ and $H_n$ is the usual copy of $H_n$.

Given $g\in E_n$, we have $\omega(g)(k,i)=\xi^{-1}(g(k+(i-1)n))$. For sufficiently large $i$, we have $\omega(g)(k,i+1)=\xi^{-1}(g(k+in))=\xi^{-1}(g(k+(i-1)n)+n)$. This shows that for sufficiently large $i$, if $\omega(g)(k,i)=(\ell,j)$ then $\omega(g)(k,i+1)=(\ell,j+1)$. Hence, $\omega(g)$ acts on $[n]\times\N$ by permutations in the first entry and eventual translations in the second entry, so $\omega(g)\in S_n\ltimes H_n$. This shows that $\omega$ sends $E_n$ into $S_n\ltimes H_n$, and it is easy to see that the image of $E_n$ is all of $S_n\ltimes H_n$, so we conclude that $\omega$ restricts to an isomorphism $E_n\to S_n \ltimes H_n$.
\end{proof}

It is also worth mentioning that $S_n\ltimes H_n$ is isomorphic to the automorphism group $\Aut(H_n)$ \cite{burillo16}.

We can now show that the Houghton groups themselves admit shift-similar representations:

\begin{corollary}\label{cor:houghton_shift_sim}
The Houghton group $H_n$ is shift-similar.
\end{corollary}

\begin{proof}
Under the isomorphism $\omega\colon \Sym(\N)\to \Sym([n]\times\N)$ from the proof of Lemma~\ref{lem:even_per_houghton} (or really its inverse), we see that $H_n$ is isomorphic to the subgroup of $\Sym(\N)$ consisting of all permutations that are eventually $n$-periodic and that eventually preserve residue classes modulo $n$, that is, all $g$ such that for sufficiently large $i$ we have $g(i+n)=g(i)+n$ and $g(i)\equiv_n i$. These properties are clearly preserved under shifting maps.
\end{proof}

\begin{example}[Shift-similar closure]\label{ex:closure}
One last immediate example is, given any $G\le \Sym(\N)$ let $\overline{G}\le \Sym(\N)$ be the smallest group containing $G$ and satisfying $\shift_j(g)\in \overline{G}$ for all $j\in\N$ and $g\in \overline{G}$. Let us call this the \emph{shift-similar closure} of $G$. In practice it seems quite difficult to analyze shift-similar closures, and for now we do not have much to say about them, but they seem like they could be potentially useful for constructing pathological examples.
\end{example}

\subsection{Group of germs at infinity}\label{ssec:germs}

The following is an important general result about shift-similar groups, which will lead us to the notion of the group of germs at infinity of an infinite shift-similar group.

\begin{theorem}\label{thrm:contain_symfin}
If $G\le\Sym(\N)$ is any infinite shift-similar group, then $\Symfin(\N)\le G$.
\end{theorem}

\begin{proof}
We need to show that $G$ contains $S_m=\Sym(\{1,\dots,m\})$ for every $m\in\N$. We induct on $m$. The base case $m=1$ holds trivially, so let $m>1$ and assume $G$ contains $S_{m-1}$. It now suffices to show that $G$ contains the transposition $(m-1~m)$.

We first claim that there exist $i,j\in\N$ and $g\in G$ such that $m-1\le i<j$ and $m-1\le g(j)<g(i)$. First suppose every element of $G$ stabilizes $\{1,\dots,m-1\}$, and hence stabilizes $\{m,m+1,\dots\}$. Since $G$ is infinite, some $g\in G$ must act non-trivially on $\{m,m+1,\dots\}$, and thus we can choose our $i$ and $j$ from $\{m,m+1,\dots\}$. Now suppose some element of $G$ does not stabilize $\{1,\dots,m-1\}$. By the pigeonhole principle, some element $g$ of $G$ must simultaneously send some $j\in \{m,m+1,\dots\}$ into $\{1,\dots,m-1\}$ and also send some $i\in \{1,\dots,m-1\}$ into $\{m,m+1,\dots\}$. Since $G$ contains $S_{m-1}$, without loss of generality $i=m-1$ and $g(j)=m-1$. Now these $i$, $j$, and $g$ satisfy all the desired properties.

At this point we have $i,j\in\N$ and $g\in G$ such that $m-1\le i<j$ and $m-1\le g(j)<g(i)$, and we are trying to prove that $(m-1~m)\in G$. Choose $k\in\N$ such that each of $k$ and $g(k)$ is larger than all of $i$, $j$, $g(i)$, and $g(j)$. Then $\shift_j(g)$ sends $i$ to $g(i)-1$ and $\shift_k(g)$ sends $i$ to $g(i)$, and also note that $\shift_j(g)$ and $\shift_k(g)$ agree on every input larger than $\max\{k,g(k)\}$. In particular $h\coloneqq \shift_j(g)\circ \shift_k(g)^{-1} \in G$ is a non-trivial element of $\Symfin(\N)$. Moreover, since $j$, $g(j)$, $k$, and $g(k)$ are all at least $m-1$, we have that $h$ fixes all of $\{1,\dots,m-2\}$. Now hitting $h$ with a finite sequence of appropriate shifting maps, we can produce $(m-1~m)$, and so we conclude $(m-1~m)\in G$ as desired.
\end{proof}

This result imposes some restrictions on shift-similar groups, which we now point out. Recall that a group \emph{virtually} has a property if it has a finite index subgroup with the property.

\begin{corollary}\label{cor:non_properties}
Infinite shift-similar groups cannot have any of the following properties, even virtually: torsion-free, solvable, hyperbolic, residually finite, or being a $p$-group.
\end{corollary}

\begin{proof}
All these properties other than hyperbolicity are inherited by subgroups, and $\Symfin(\N)$ does not have any of these properties, hence by Theorem~\ref{thrm:contain_symfin} neither does any infinite shift-similar group. For (virtual) hyperbolicity, any torsion subgroup of a hyperbolic group is finite (see, e.g., \cite[Theorem~1(c)]{juriaans05}), so no virtually hyperbolic group can contain $\Symfin(\N)$.
\end{proof}

Now we can define the group of germs at infinity. We should reiterate that $\Symfin(\N)$ is normal in $\Sym(\N)$, since conjugating a permutation translates its support and hence preserves finiteness of support, and so $\Symfin(\N)$ is normal in every infinite shift-similar group.

\begin{definition}[Group of germs at infinity]\label{def:germs}
For $G\le\Sym(\N)$ an infinite shift-similar group, define the \emph{group of germs at infinity} for $G$ to be the quotient group
\[
\Germs(G)\coloneqq G/\Symfin(\N) \text{.}
\]
\end{definition}

One should think of an element of $\Germs(G)$ as encoding ``eventual behavior'' of an element of $G$. In particular two elements of $G$ represent the same germ at infinity if and only if they agree on the complement of some finite subset.

\begin{lemma}
Let $G\le \Sym(\N)$ be an infinite shift-similar group. For any $j\in\N$, the function $\shift_j\colon G\to G$ induces an injective homomorphism
\[
\shift_\infty \colon \Germs(G) \to \Germs(G) \text{,}
\]
which is independent of $j$, defined via
\[
\shift_\infty(g\Symfin(\N)) \coloneqq \shift_j(g)\Symfin(\N) \text{.}
\]
\end{lemma}

\begin{proof}
First note that for any $i,j\in\N$ and any $g\in G$ and $f\in \Symfin(\N)$, we have that $\shift_i(g)$ and $\shift_j(g\circ f)$ agree on all sufficiently large inputs. This shows that $\shift_\infty$ is well defined and independent of $j$. Now we claim that it is a homomorphism. Indeed, for all $g,h\in G$, using Observation~\ref{obs:almost_hom} we have
\begin{align*}
\shift_\infty((g\circ h)\Symfin(\N)) &= \shift_1(g\circ h)\Symfin(\N)\\
&= (\shift_{h(1)}(g)\circ \shift_1(h))\Symfin(\N)\\
&= (\shift_{h(1)}(g)\Symfin(\N))(\shift_1(h)\Symfin(\N))\\
&= \shift_\infty(g\Symfin(\N))\shift_\infty(h\Symfin(\N)) \text{.}
\end{align*}
Finally, we need to show $\shift_\infty$ is injective. Suppose $g\in G$ with $\shift_1(g)\in \Symfin(\N)$, so $\shift_1(g)$ fixes all sufficiently large $i\in\N$. For sufficiently large $i$ we have $\shift_1(g)(i)+1=g(i+1)$, so in fact $g$ must also fix all sufficiently large inputs, i.e., $g\in\Symfin(\N)$.
\end{proof}

\begin{definition}[Germ shifting map]
Call the monomorphism $\shift_\infty \colon \Germs(G) \to \Germs(G)$ the \emph{germ shifting map} for an infinite shift-similar group $G$.
\end{definition}

While the germ shifting map is always injective, it is not clear to us whether it is always surjective. We see no reason why it should always be the case, but we do not know of an example where it is not. Let us encode this into the following definition and question:

\begin{definition}[Strongly shift-similar]\label{def:strong}
Call an infinite shift-similar group $G\le\Sym(\N)$ \emph{strongly shift-similar} if $\shift_\infty$ is surjective.
\end{definition}

\begin{question}\label{quest:not_strong}
Does there exist an infinite shift-similar group that is not strongly shift-similar?
\end{question}

Informally, shift-similarity asks that when we ``delete an arrow'' from an element of $G$ (in the sense of Figure~\ref{fig:shift}) then we get an element of $G$. Strong shift-similarity asks that, additionally, when we ``add a new arrow'' to an element of $G$ then, up to elements of finite support, we get an element of $G$.

This definition involving surjectivity of $\shift_\infty$ and the groups of germs at infinity will be useful in what follows, since it is often convenient to be able to work modulo $\Symfin(\N)$, but it is worth pointing out that it is equivalent to something easier to think about, namely surjectivity of the $\shift_j$:

\begin{lemma}\label{lem:strong_shift1_surj}
If $G\le\Sym(\N)$ is a strongly shift-similar group, then the function $\shift_j\colon G\to G$ is surjective for all $j\in\N$. In particular, for any $j\in\N$, the restriction
\[
\shift_j|_{\Stab_G(j)}\colon \Stab_G(j)\to G
\]
is an isomorphism.
\end{lemma}

\begin{proof}
We are assuming that $\shift_\infty\colon \Germs(G)\to\Germs(G)$ is surjective. Let $g\in G$. By surjectivity of $\shift_\infty$, we can choose $g'\in G$ and $f\in \Symfin(\N)$ such that $g=f\circ \shift_j(g')$. Let $f'\in \Symfin(\N)$ be such that $f=\shift_{g'(j)}(f')$ and $f'(g'(j))=j$; informally, we add a new arrow to $f$, from a new domain point at $g'(j)$ to a new range point at $j$, and it is clear $f'$ still has finite support. Now by Observation~\ref{obs:almost_hom} we have
\[
\shift_j(f'\circ g')=\shift_{g'(j)}(f')\circ \shift_j(g') = f\circ \shift_j(g') = g \text{,}
\]
and so $\shift_j$ is surjective. In fact, since $f'\circ g'$ fixes $j$, this shows that the restriction of $\shift_j$ to $\Stab_G(j)$ already surjects onto $G$. Finally, to see that this restriction is an isomorphism, note that it is a homomorphism by Observation~\ref{obs:almost_hom}, and if $\shift_j(g)=1$ for $g\in\Stab_G(j)$ then $g=1$.
\end{proof}

Note that this shows $\Stab_G(j)\cong G$ for any $j\in\N$, which is interesting in its own right.

The concrete examples of infinite shift-similar groups we have seen so far are $\Symfin(\N)$, $E_n$, and $H_n$, and it is easy to see these are all strongly shift-similar. For example, $\Germs(H_n)\cong \Z^{n-1}$, which one should think of as the copy of $\Z^{n-1}$ sitting inside $\Z^n$ as the subgroup of all $(m_1,\dots,m_n)$ with $m_1+\cdots+m_n=0$, and one can work out that $\psi_\infty$ is described by the automorphism cyclically permuting the coordinates, since for example shifting $\N$ by $1$ takes the $k$th residue class modulo $n$ to the $(k+1)$st residue class modulo $n$. Things are less clear for the last example we have seen so far, namely the shift-similar closure $\overline{G}$ of a subgroup $G\le \Symfin(\N)$ (see Example~\ref{ex:closure}). We do not see a reason why these should always be strongly shift-similar, especially, for example, if $G$ is cyclic, but it seems difficult in general to analyze groups of the form $\overline{G}$, and we do not know of an example where $\overline{G}$ is not strongly shift-similar.

\subsection{Sufficient conditions for shift-similarity}\label{ssec:suff_cond}

Before continuing to populate the list of examples of shift-similar groups, let us discuss some tools for proving shift-similarity. In practice, when proving that a group is shift-similar it can be convenient to only have to check $\shift_j(g)\in G$ for certain $j$ and/or $g$. In this subsection we establish some sufficient conditions for shift-similarity in this vein.

First we have a sufficient condition that is useful if we have some nice generating set.

\begin{lemma}\label{lem:gens_suffice}
Let $G\le \Sym(\N)$ and let $A\subseteq G$ be a generating set. If $\shift_j(a)\in G$ for all $j\in\N$ and all $a\in A$, then $G$ is shift-similar.
\end{lemma}

\begin{proof}
By Corollary~\ref{cor:inverse} we can assume without loss of generality that $a^{-1}\in A$ for all $a\in A$. Now the result follows from Observation~\ref{obs:almost_hom} since any element of $G$ is a product of elements of $A$.
\end{proof}

Next we have a sufficient condition that is useful if $G$ is already known to be transitive on $\N$:

\begin{lemma}\label{lem:trans}
Let $G\le \Sym(\N)$ be transitive on $\N$. If there exists $j_0\in\N$ such that $\shift_{j_0}(g)\in G$ for all $g\in G$, then $G$ is shift-similar.
\end{lemma}

\begin{proof}
Let $j\in\N$ and $g\in G$. By transitivity we can choose $h_1,h_2\in G$ such that $h_1(g(j))=j_0$ and $h_2(j_0)=j$. We know that $\shift_{j_0}(h_1\circ g\circ h_2)\in G$ by hypothesis. By Observation~\ref{obs:almost_hom}, this equals
\[
\shift_{j_0}(h_1\circ g\circ h_2) = \shift_{g\circ h_2(j_0)}(h_1) \circ \shift_{h_2(j_0)}(g) \circ \shift_{j_0}(h_2)\text{.}
\]
Since $\shift_{j_0}(h_2)\in G$ and $h_2(j_0)=j$, this shows that $\shift_{g(j)}(h_1) \circ \shift_j(g)\in G$. Also, by Corollary~\ref{cor:inverse} we have
\[
\shift_{g(j)}(h_1)^{-1} = \shift_{h_1(g(j))}(h_1^{-1}) = \shift_{j_0}(h_1^{-1}) \in G \text{,}
\]
so we conclude that $\shift_j(g)\in G$ as desired.
\end{proof}

Finally, we have a sufficient condition that is useful if we already know $\Symfin(\N)\le G$.

\begin{lemma}\label{lem:stab_suffices}
Let $\Symfin(\N)\le G\le \Sym(\N)$. If $\shift_1(h)\in G$ for all $h\in \Stab_G(1)$ then $G$ is shift-similar.
\end{lemma}

\begin{proof}
By Lemma~\ref{lem:trans}, which applies since $\Symfin(\N)$ is transitive on $\N$, it suffices to show that $\shift_1(g)\in G$ for all $g\in G$. Let $j=g(1)$ and consider the cycle $c=(1~2\cdots j)\in \Symfin(\N)$. Note that $c\circ g \in \Stab_G(1)$, so $\shift_1(c\circ g)\in G$ by hypothesis. By Observation~\ref{obs:almost_hom} we have $\shift_1(c\circ g)=\shift_j(c)\circ \shift_1(g)$. But $\shift_j(c)$ is the identity, so $\shift_1(g)\in G$, as desired.
\end{proof}

\subsection{A class of examples}\label{ssec:containing_E2}

In this subsection, we discuss a class of examples of shift-similar groups arising from groups containing $E_2$. Here $E_2$ is the group of eventually $2$-periodic permutations from Example~\ref{ex:eventually_periodic}, which is isomorphic to $S_2\ltimes H_2$. These examples will show that shift-similar groups are quite abundant (see Theorem~\ref{thrm:embed_fin_gen} and Corollary~\ref{cor:uncountable}).

\begin{definition}[Slide and flip]
We will denote by $\slide$ the element of $E_2$ given by
\[
\slide\coloneqq (\cdots 6~4~2~1~3~5~7\cdots) \text{,}
\]
and call $\slide$ the \emph{slide}. Thus $\slide$ sends each odd number to the next odd number, each even number besides 2 to the previous even number, and 2 to 1. The image of $\slide$ under the isomorphism $\omega\colon E_2\to S_2\ltimes H_2$ is the element of $H_2$ sending each $(1,i)$ to $(1,i+1)$, each $(2,i)$ to $(2,i-1)$ for $i\ge 2$, and $(2,1)$ to $(1,1)$.

Also denote by $\flip$ the element of $E_2$ given by
\[
\flip\coloneqq (1~2)(3~4)(5~6)\cdots \text{,}
\]
and call $\flip$ the \emph{flip}. The image of $\flip$ under the isomorphism $\omega\colon E_2\to S_2\ltimes H_2$ is the element $((1~2),\id)$.
\end{definition}

See Figure~\ref{fig:slide_flip} for a helpful visualization of $\slide$ and $\flip$.

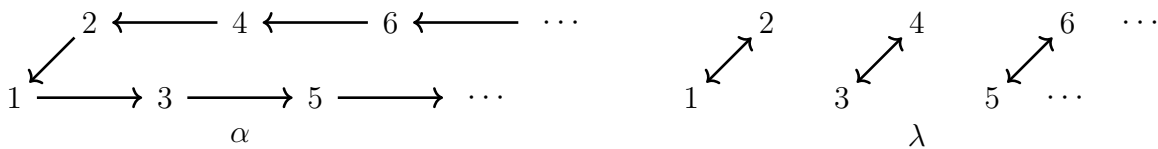
\begin{figure}[htb]\centering
 \begin{tikzpicture}[line width=1pt]
    \draw[->] (1+.3,0) -- (3-.3,0);
    \draw[->] (2-.2,1-.2) -- (1+.2,0+.2);
    \draw[->] (3+.3,0) -- (5-.3,0);
    \draw[->] (4-.3,1) -- (2+.3,1);
    \draw[->] (5+.3,0) -- (7-.3,0);
    \draw[->] (6-.3,1) -- (4+.3,1);
    \draw[->] (8-.3,1) -- (6+.3,1);
    
    \node at (1,0) {$1$};
    \node at (3,0) {$3$};
    \node at (5,0) {$5$};
    \node at (2,1) {$2$};
    \node at (4,1) {$4$};
    \node at (6,1) {$6$};
    
    \node at (8+.3,1) {$\cdots$};
    \node at (7+.3,0) {$\cdots$};
    
    \node at (4,-.5) {$\slide$};
    
   \begin{scope}[xshift=9cm]
    \draw[<->] (1+.2,0+.2) -- (2-.2,1-.2);
    \draw[<->] (3+.2,0+.2) -- (4-.2,1-.2);
    \draw[<->] (5+.2,0+.2) -- (6-.2,1-.2);
    
    \node at (1,0) {$1$};
    \node at (3,0) {$3$};
    \node at (5,0) {$5$};
    \node at (2,1) {$2$};
    \node at (4,1) {$4$};
    \node at (6,1) {$6$};
    
    \node at (7,1) {$\cdots$};
    \node at (6,0) {$\cdots$};
    
    \node at (4,-.5) {$\flip$};
   \end{scope}
 \end{tikzpicture}
\caption{A good way to picture $\slide$ and $\flip$. We lay out $\N$ from left to right, but in a zigzag, so that the arrows describing $\slide$ are more visibly encoding a ``slide'' move.}\label{fig:slide_flip}
\end{figure}

Enumerating $\Z$ as $\{0,-1,1,-2,2,-3,3,\dots\}$ and identifying this with $\N$, we can view $\slide$ and $\flip$ as very straightforward elements of $\Sym(\Z)$, namely $n\mapsto n+1$ and $n\mapsto -n-1$ respectively. This is a nice way to picture these elements, but using $\Z$ makes it harder to deal with issues of shift-similarity, so we will stick to working in $\Sym(\N)$.

\begin{observation}\label{obs:E2_fin_gen}
The group $E_2$ is generated by $\slide$, $\flip$, and $(1~2)$. In particular $E_2$ is finitely generated.
\end{observation}

\begin{proof}
Hitting everything with the isomorphism $\omega\colon E_2\to S_2\ltimes H_2$, it is clear that $\omega(\flip)$ generates the $S_2$ factor and it is well known that $\omega(\slide)$ and $\omega(1~2)$ generate the $H_2$ factor, see, e.g., \cite[Figure~2.2]{lee12}.
\end{proof}

\begin{definition}[Half-finite support]
A permutation $g\in \Sym(\N)$ has \emph{half-finite support} if $g(i)=i$ either for all sufficiently large even $i$, or for all sufficiently large odd $i$.
\end{definition}

Note that a permutation satisfies both of these rules if and only if it has finite support, but it is easy to construct permutations with infinite, half-finite support, e.g., $(1~3)(5~7)(9~11)\cdots$.

\begin{lemma}\label{lem:shift_half_finite}
Let $g\in\Sym(\N)$ have half-finite support. Then for any $j\in\N$, we have that $\shift_j(g)$ equals either $g^{\slide\circ\flip} \circ f$ or $g^{\flip}\circ f$ for some $f\in\Symfin(\N)$.
\end{lemma}

Here the notation $x^y$ for $x$ and $y$ group elements denotes conjugation, $x^y=y^{-1}xy$.

\begin{proof}
The claim is that the element $\shift_j(g)$ either does the same thing as $g^{\slide\circ\flip}$ to all sufficiently large $i\in\N$, or does the same thing as $g^{\flip}$ to all sufficiently large $i$. Note that for all sufficiently large $i$, we know that $\shift_j(g)$ sends $i$ to $g(i+1)-1$.

First suppose $g(i)=i$ for all sufficiently large even $i$. This means that for all sufficiently large odd $i$, $g(i)$ is also odd. We claim that $g^{\slide\circ\flip}=(\slide\circ\flip)^{-1}\circ g\circ (\slide\circ\flip)$ sends sufficiently large $i$ to $g(i+1)-1$, which will finish this case. First consider sufficiently large even $i$. Then we have that $(\slide\circ\flip)^{-1}\circ g\circ (\slide\circ\flip)$ sends $i$ to $g(i+1)-1$, as desired. For sufficiently large odd $i$, this element sends $i$ to $g(i-1)+1$, and since $g$ fixes all sufficiently large even inputs (such as $i+1$ and $i-1$), this is the same as $g(i+1)-1$ as desired.

Now suppose $g(i)=i$ for all sufficiently large odd $i$, so for sufficiently large even $i$, $g(i)$ is also even. We claim that $g^{\flip}=\flip^{-1}\circ g\circ \flip$ sends sufficiently large $i$ to $g(i+1)-1$, which will finish this case. First consider sufficiently large odd $i$. Then we have that $\flip^{-1}\circ g\circ \flip$ sends $i$ to $g(i+1)-1$, as desired. For sufficiently large even $i$, this element sends $i$ to $g(i-1)+1$, and since $g$ fixes all sufficiently large odd inputs (such as $i+1$ and $i-1$), this is the same as $g(i+1)-1$ as desired.
\end{proof}

A composition of permutations with half-finite support need not have half-finite support, so the permutations with half-finite support do not form a group. However, one can consider groups generated by permutations with half-finite support, and these turn out to yield shift-similar groups in the following sense:

\begin{proposition}\label{prop:half_fin_plus_E2}
Let $\Gamma\le\Sym(\N)$ be any group generated by permutations with half-finite support. Let $G=\langle \Gamma,E_2\rangle$. Then $G$ is strongly shift-similar.
\end{proposition}

\begin{proof}
First we prove shift-similarity. We already know that $\shift_j(e)\in E_2$ for all $e\in E_2$ by Example~\ref{ex:eventually_periodic}, so by Lemma~\ref{lem:gens_suffice} we just need to check that for any $j\in\N$ and any generator $\gamma$ of $\Gamma$, we have $\shift_j(\gamma)\in G$. Since $\gamma$ has half-finite support, by Lemma~\ref{lem:shift_half_finite} we have that $\shift_j(\gamma)$ is a conjugate of $\gamma$ by $\slide\circ\flip$ or $\flip$, multiplied by an element of $\Symfin(\N)$. Since $E_2$ contains $\slide$, $\flip$, and $\Symfin(\N)$, we conclude that $\shift_j(\gamma) \in G$.

Now we prove strong shift-similarity. Since $E_2$ is strongly shift-similar, it suffices to prove that for any generator $\gamma$ of $\Gamma$, the image of the germ shifting map $\shift_\infty \colon \Germs(G)\to\Germs(G)$ contains $\gamma\Symfin(\N)$. By Lemma~\ref{lem:shift_half_finite}, we can choose $g\in G$ to be one of $\gamma^{\flip^{-1}\circ\slide^{-1}}$ or $\gamma^{\flip^{-1}}$, and get $\shift_1(g)=\gamma\circ f$ for some $f\in\Symfin(\N)$. Thus $\shift(g\Symfin(\N))=\gamma\Symfin(\N)$ and so indeed the image of $\shift$ contains $\gamma\Symfin(\N)$.
\end{proof}

Since $E_2$ is finitely generated (Observation~\ref{obs:E2_fin_gen}), it is now very easy to construct examples of finitely generated strongly shift-similar groups: Let $\gamma_1,\dots,\gamma_k$ be any elements of $\Sym(\N)$ with half-finite support. Then $G=\langle \gamma_1,\dots,\gamma_k,E_2\rangle$ is a finitely generated strongly shift-similar group. Indeed, this shows that finitely generated strongly shift-similar groups abound, in the following sense:

\begin{theorem}\label{thrm:embed_fin_gen}
For any finitely generated group $\Gamma$, there exists a finitely generated strongly shift-similar group $G\le\Sym(\N)$ such that $\Gamma$ embeds as a subgroup of $G$.
\end{theorem}

\begin{proof}
First note that every finitely generated group embeds as a subgroup of $\Sym(\N)$, and hence as a subgroup of $\Sym(2\N)$ (since $\Sym(\N)\cong \Sym(2\N)$). Viewing $\Sym(2\N)$ as a subgroup of $\Sym(\N)$, every permutation in $\Sym(2\N)$ has half-finite support, so we conclude that every finitely generated group $\Gamma$ embeds as a subgroup of $\Sym(\N)$ in such a way that every generator (indeed every element) has half-finite support. Now the group $G$ generated by this image together with $E_2$ is a finitely generated strongly shift-similar group, into which $\Gamma$ embeds.
\end{proof}

\begin{corollary}\label{cor:uncountable}
There exist uncountably many isomorphism classes of finitely generated strongly shift-similar groups.
\end{corollary}

\begin{proof}
Suppose there are only countably many isomorphism classes of finitely generated strongly shift-similar groups. Since any given finitely generated group contains only countably many (isomorphism classes of) finitely generated subgroups, Theorem~\ref{thrm:embed_fin_gen} implies that there exist countably many isomorphism classes of finitely generated groups. But it is well known that there exist uncountably many isomorphism classes of finitely generated groups, so this is a contradiction.
\end{proof}

As we have already pointed out, this contrasts with the self-similar situation, where there exist only countably many isomorphism classes of finitely generated self-similar groups \cite[Subsection~1.5.3]{nekrashevych05}.

Another consequence is the following:

\begin{corollary}\label{cor:undecidable}
There exist finitely generated strongly shift-similar groups with undecidable word problem.
\end{corollary}

\begin{proof}
By Theorem~\ref{thrm:embed_fin_gen}, any finitely generated group with undecidable word problem embeds in some finitely generated strongly shift-similar group, which therefore must also have undecidable word problem.
\end{proof}

Proposition~\ref{prop:half_fin_plus_E2} is concrete enough to allow us to produce tons of concrete examples of finitely generated shift-similar groups. Indeed, we can simply take any finite set of permutations $\gamma_1,\dots,\gamma_k$ of $2\N$, and take the (finitely generated) group of permutations of $\N$ generated by them together with $\alpha$, $\lambda$, and $(1~2)$. Let us record a more precise way of constructing examples of interesting finitely generated shift-similar groups, which could warrant further future investigation.

\begin{example}\label{ex:lots_of_fg_ss}
Let $\Gamma$ be any finitely generated infinite group, and fix a bijection $\nu\colon \Gamma\to \N$. Identify the disjoint union $\Gamma\coprod\Gamma$ with $\N$ by sending the first copy of $\Gamma$ to $2\N$ via $\gamma\mapsto 2\nu(\gamma)$ and sending the second copy of $\Gamma$ to $2\N-1$ via $\gamma\mapsto 2\nu(\gamma)-1$. Now let $G$ be the group of permutations of $\N$ generated by the copy of $\Gamma$ acting on $2\N$ by left multiplication (via the above identification of $\Gamma$ with $2\N$) and trivially on $2\N-1$, and the elements $\alpha$, $\lambda$, and $(1~2)$. By Proposition~\ref{prop:half_fin_plus_E2}, the group $G$ is finitely generated and shift-similar.

More intuitively, if we use $\nu$ to identify $\Gamma\coprod\Gamma$ with $-\N\coprod \N$, where the first copy of $\Gamma$ is identified with $-\N$ and the other with $\N$, then $G$ is the group of permutations of $\Gamma\coprod\Gamma$ where we can do any finitely supported permutation, we can swap the two copies of $\Gamma$ in the canonical way (this is the element $\lambda$), we can have $\Gamma$ act on either copy of itself by translation while fixing the other copy pointwise, and we can ``shift'' left and right viewing $\Gamma\coprod\Gamma$ as $-\N\coprod \N$ (this is the element $\alpha$). We emphasize that $\nu$ can be any bijection whatsoever. It would be especially interesting to investigate whether $G$ could possibly be finitely presented, for some $\Gamma$ of interest and some cleverly chosen $\nu$.
\end{example}

\section{Houghton-like groups from shift-similar groups}\label{sec:houghton_like}

In this section, we introduce our groups $H_n(G)$, for $G$ a shift-similar group. These are a sort of ``Houghton-world'' analog of R\"over--Nekrashevych groups. We will once again use the notions of quasi-ray and canonical bijection from Section~\ref{sec:houghton}.

\begin{definition}[The group $H_n(G)$]\label{def:houghton_like}
Let $G\le\Sym(\N)$ be shift-similar. For $n\in \N$, let $H_n(G)$ be the group of bijections from $[n]\times\N$ to itself given as follows:
\begin{enumerate}
    \item Take a partition of $[n]\times\N$ into a finite subset $M_+$ and the corresponding quasi-rays $Q(1,M_+),\dots,Q(n,M_+)$.
    \item Take another partition of $[n]\times\N$ into a finite subset $M_-$ with $|M_-|=|M_+|$ and the corresponding quasi-rays $Q(1,M_-),\dots,Q(n,M_-)$.
    \item Map $[n]\times\N$ to itself by sending $M_+$ to $M_-$ via some bijection $\sigma$, and for each $1\le k\le n$ sending $Q(k,M_+)$ to $Q(k,M_-)$ via $\beta_{k,M_-}\circ g_k \circ \beta_{k,M_+}^{-1}$ for some $g_k\in G$.
\end{enumerate}
\end{definition}

Note that the only difference between this and the definition of $H_n$ is the presence of the elements $g_k\in G$ at the end. Intuitively, in $H_n$ we map one quasi-ray to another quasi-ray by identifying them both with $\N$ via canonical bijections and then mapping $\N$ to itself via the identity, whereas in $H_n(G)$ we do the same thing but map $\N$ to itself via an element of $G$. In particular $H_n=H_n(\{1\})$. We will prove later that $H_n(G)$ really is a group. First it is convenient to establish representative triples.

We can define representative triples for elements of $H_n(G)$ by essentially merging the ideas behind representative triples for $V_d(G)$ and for $H_n$. By a \emph{representative triple} now we mean a triple $(M_-,\sigma(g_1,\dots,g_n),M_+)$, where $M_+$ and $M_-$ are finite subsets of $[n]\times\N$ of the same size, $\sigma$ is a bijection from $M_+$ to $M_-$, and $g_1,\dots,g_n\in G$. At the moment, the notation $\sigma(g_1,\dots,g_n)$ should just be thought of as an ordered pair $(\sigma,(g_1,\dots,g_n))$; we use this notation to continue the analogy with R\"over--Nekrashevych groups, and later when we discuss multiplication in $H_n(G)$ we will explain how to manipulate the notation. The element of $H_n(G)$ represented by $(M_-,\sigma(g_1,\dots,g_n),M_+)$ is the one obtained via the procedure in Definition~\ref{def:houghton_like} using $M_+$, $M_-$, $\sigma$, and $g_1,\dots,g_n$.

As in the previous situations, we can define a notion of expansion:

\begin{definition}[Expansion/reduction/equivalence]
Let $G\le\Sym(\N)$ be shift-similar. Let $(M_-,\sigma(g_1,\dots,g_n),M_+)$ be a representative triple for an element of $H_n(G)$. Let $1\le k\le n$. The $k$th \emph{expansion} of this triple is the triple
\[
(M_-\cup\{\beta_{k,M_-}(g_k(1))\},\sigma'(g_1,\dots,g_{k-1},\shift_1(g_k),g_{k+1},\dots,g_n),M_+\cup\{\beta_{k,M_+}(1)\})\text{,}
\]
where $\sigma'$ is the bijection from $M_+\cup\{\beta_{k,M_+}(1)\}$ to $M_-\cup\{\beta_{k,M_-}(g_k(1))\}$ sending $\beta_{k,M_+}(1)$ to $\beta_{k,M_-}(g_k(1))$ and otherwise acting like $\sigma$. If one triple is an expansion of another, call the second a \emph{reduction} of the first. Call two representative triples \emph{equivalent} if one can be obtained from the other by a finite sequence of expansions and reductions. Write $[M_-,\sigma(g_1,\dots,g_n),M_+]$ for the equivalence class of $(M_-,\sigma(g_1,\dots,g_n),M_+)$.
\end{definition}

We also have an $H_n(G)$ analog of general expansions:

\begin{definition}[General expansion/reduction]
Let $(M_-,\sigma(g_1,\dots,g_n),M_+)$ be a representative triple. The \emph{general expansion} of this triple associated to the pair $(k,j)$ is the triple
\[
(M_-\cup\{\beta_{k,M_-}(g_k(j))\},\sigma'(g_1,\dots,g_{k-1},\shift_j(g_k),g_{k+1},\dots,g_n),M_+\cup\{\beta_{k,M_+}(j)\})\text{,}
\]
where $\sigma'$ is the bijection from $M_+\cup\{\beta_{k,M_+}(j)\}$ to $M_-\cup\{\beta_{k,M_-}(j)\}$ sending $\beta_{k,M_+}(j)$ to $\beta_{k,M_-}(j)$ and otherwise acting like $\sigma$. The reverse operation is called a \emph{general reduction}.
\end{definition}

We should emphasize that defining expansions and general expansions is the key step in all this setup where $G$ needs to be shift-similar, so that $\shift_j(g_k)$ is still an element of $G$.

The following is the $H_n(G)$ version of Lemma~\ref{lem:gen_exp}, and the structure of the proof is the same.

\begin{lemma}\label{lem:gen_exp_with_G}
If one representative triple is a general expansion of the other, then they are equivalent.
\end{lemma}

\begin{proof}
Say the first triple is $(M_-,\sigma(g_1,\dots,g_n),M_+)$ and the second is
\[
(M_-\cup\{\beta_{k,M_-}(g_k(j))\},\sigma'(g_1,\dots,g_{k-1},\shift_j(g_k),g_{k+1},\dots,g_n),M_+\cup\{\beta_{k,M_+}(j)\})
\]
as above. Let us induct on $j$. If $j=1$ then this is just an expansion, and we are done. Now assume $j>1$. By induction $(M_-,\sigma(g_1,\dots,g_n),M_+)$ is equivalent via a general expansion to
\[
(M_-\cup\{\beta_{k,M_-}(g_k(j-1))\},\sigma''(g_1,\dots,g_{k-1},\shift_{j-1}(g_k),g_{k+1},\dots,g_n),M_+\cup\{\beta_{k,M_+}(j-1)\})\text{,}
\]
for appropriate $\sigma''$. Now observe that $\beta_{k,M_+\cup\{\beta_{k,M_+}(j-1)\}}(j-1)$ equals $\beta_{k,M_+}(j)$ (with a similar statement for the $M_-$ version), so again by induction our triple is equivalent via a general expansion to
\begin{align*}
(M_-\cup\{\beta_{k,M_-}(g_k(j-1)),\beta_{k,M_-}(g_k(j))\},&\\
\sigma'''(g_1,\dots,g_{k-1},\shift_{j-1}&\circ\shift_{j-1}(g_k),g_{k+1},\dots,g_n),\\
&M_+\cup\{\beta_{k,M_+}(j-1),\beta_{k,M_+}(j)\})\text{,}
\end{align*}
for appropriate $\sigma'''$. Next we note that $\beta_{k,M_+\cup\{\beta_{k,M_+}(j)\}}(j-1)$ equals $\beta_{k,M_+}(j-1)$ (with a similar statement for the $M_-$ version), and $\shift_{j-1}\circ \shift_{j-1}=\shift_{j-1}\circ \shift_j$ by Lemma~\ref{lem:two_shifts}, so again by induction (now doing a general reduction) our triple is equivalent to
\[
(M_-\cup\{\beta_{k,M_-}(g_k(j))\},\sigma''''(g_1,\dots,g_{k-1},\shift_j(g_k),g_{k+1},\dots,g_n),M_+\cup\{\beta_{k,M_+}(j)\})
\]
for appropriate $\sigma''''$. It is clear that $\sigma''''=\sigma'$, since they both agree with $\sigma$ on $M_+$ and their domain has only one other point, so we are done.
\end{proof}

We remark that, just like in the $H_n$ case (where $G=\{1\}$), given any two representative triples $(M_-,\sigma(g_1,\dots,g_n),M_+)$ and $(N_-,\tau(h_1,\dots,h_n),N_+)$, up to equivalence we can assume that one of the finite sets $M_-$ or $M_+$ in the first triple equals one of the finite sets $N_-$ or $N_+$ in the second triple, and we can pick which ones we want to be equal.

Now we can establish the analog of Lemma~\ref{lem:triples} for $H_n(G)$, which lets us view $H_n(G)$ as the group of equivalence classes of representative triples.

\begin{lemma}\label{lem:triples_with_G}
Two representative triples are equivalent if and only if the elements of $H_n(G)$ they represent are equal.
\end{lemma}

\begin{proof}
Let $\eta,\theta\in H_n(G)$ have equivalent representative triples, and we want to show $\eta=\theta$. Without loss of generality the triples differ by a single expansion, say there is a representative triple of $\theta$ that is an expansion of one for $\eta$. Say $\eta$ has $(M_-,\sigma(g_1,\dots,g_n),M_+)$ as a representative triple and $\theta$ has $(M_-',\sigma'(g_1,\dots,g_{k-1},\shift_1(g_k),g_{k+1},\dots,g_n),M_+')$, where $M_+'=M_+\cup\{\beta_{k,M_+}(1)\}$ and $M_-'=M_-\cup\{\beta_{k,M_-}(g_k(1))\}$. Now we want to show that $\theta(\ell,i)=\eta(\ell,i)$ for all $(\ell,i)\in [n]\times\N$. First suppose $(\ell,i)\in M_+$. Then since $\sigma'(\ell,i)=\sigma(\ell,i)$ we indeed have $\theta(\ell,i)=\eta(\ell,i)$. Next suppose $(\ell,i)=\beta_{k,M_+}(1)$ (so $\ell=k$). Then
\[
\theta(k,i)=\beta_{k,M_-}(g_k(1))=\beta_{k,M_-}\circ g_k \circ \beta_{k,M_+}^{-1}(k,i)=\eta(k,i)
\]
as desired. Finally suppose $(\ell,i)\in Q(\ell,M_+')$. If $\ell\ne k$ then $Q(\ell,M_+')=Q(\ell,M_+)$ and $Q(\ell,M_-')=Q(\ell,M_-)$, so
\[
\theta(\ell,i)=\beta_{\ell,M_-'}\circ g_\ell \circ \beta_{\ell,M_+'}^{-1}(\ell,i)=\beta_{\ell,M_-}\circ g_\ell \circ \beta_{\ell,M_+}^{-1}(\ell,i)=\eta(\ell,i) \text{.}
\]
Now suppose $\ell=k$, and say $(k,i)=\beta_{k,M_+'}(j)$. Note that $\beta_{k,M_+'}(j)=\beta_{k,M_+}(j+1)$ and $\beta_{k,M_-'}(\shift_1(g_k)(j))=\beta_{k,M_-}(s_{g_k(1)}\circ \shift_1(g_k)(j))$. Now we have
\begin{align*}
\theta(k,i)&=\beta_{k,M_-'}\circ \shift_1(g_k) \circ \beta_{k,M_+'}^{-1}(k,i)\\
&=\beta_{k,M_-'}(\shift_1(g_k)(j))\\
&=\beta_{k,M_-}(s_{g_k(1)}\circ \shift_1(g_k)(j))\\
&=\beta_{k,M_-}(g_k(s_1(j)))\\
&=\beta_{k,M_-}(g_k(j+1))\\
&=\beta_{k,M_-} \circ g_k \circ \beta_{k,M_+}^{-1}\circ \beta_{k,M_+'}(j)\\
&=\beta_{k,M_-} \circ g_k \circ \beta_{k,M_+}^{-1}(k,i)\\
&= \eta(k,i) \text{,}
\end{align*}
and we are done.

Now let $(M_-,\sigma(g_1,\dots,g_n),M_+)$ and $(N_-,\tau(h_1,\dots,h_n),N_+)$ be representative triples for the same element $\eta$, and we claim they are equivalent. It suffices to show that they have a common general expansion. The triples have general expansions of the form $(M_-',\sigma'(g_1',\dots,g_n'),M_+\cup N_+)$ and $(N_-',\tau'(h_1',\dots,h_n'),M_+\cup N_+)$ respectively, for some $M_-'$, $\sigma'$, $g_i'$, $N_-'$, $\tau'$, and $h_i'$. Since these both represent $\eta$ (by the first paragraph), we see that $M_-'$ and $N_-'$ both equal $\eta(M_+\cup N_+)$, hence are equal themselves. Also, $\sigma'$ and $\tau'$ each equal the restriction of $\eta$ to $M_+\cup N_+$, hence are equal. Finally, for each $1\le k\le n$ we have that $g_k'$ and $h_k'$ each equal the restriction of $\eta$ to $Q(k,M_+\cup N_+)$ composed on either side by the appropriate canonical bijections (which are the same in either case), so $g_k'=h_k'$. We conclude that $(M_-',\sigma'(g_1',\dots,g_n'),M_+\cup N_+)=(N_-',\tau'(h_1',\dots,h_n'),M_+\cup N_+)$ is a common general expansion of the original triples, as desired.
\end{proof}

Now we can finally confirm that $H_n(G)$ really is a group.

\begin{lemma}
For $G\le\Sym(\N)$ shift-similar, $H_n(G)$ is a group.
\end{lemma}

\begin{proof}
The composition of two elements of $H_n(G)$ can be computed via the following procedure. Let $[M_-,\sigma(g_1,\dots,g_n),M_+]$ and $[N_-,\tau(h_1,\dots,h_n),N_+]$ be elements. Up to performing general expansions, we can assume without loss of generality that $M_+=N_-$, so also $Q(k,M_+)=Q(k,N_-)$ for each $1\le k\le n$. Now the product
\[
[M_-,\sigma(g_1,\dots,g_n),M_+][M_+,\tau(h_1,\dots,h_n),N_+]
\]
is the self-bijection of $[n]\times\N$ given by sending $N_+$ to $N_-=M_+$ via $\tau$ and then to $M_-$ via $\sigma$, which is to say sending $N_+$ to $M_-$ via $\sigma\circ\tau$, and for each $1\le k\le n$ sending $Q(k,N_+)$ to $Q(k,N_-)=Q(k,M_+)$ via $\beta_{k,N_-}^{-1}\circ h_k \circ \beta_{k,N_+}$ and then to $Q(k,M_-)$ via $\beta_{k,M_-}^{-1}\circ g_k \circ \beta_{k,M_+}$, which since $\beta_{k,N_-}=\beta_{k,M_+}$ is the same as sending $Q(k,N_+)$ to $Q(k,M_-)$ via $\beta_{k,M_-}^{-1}\circ g_k\circ h_k \circ \beta_{k,N_+}$. Thus, we conclude that
\[
[M_-,\sigma(g_1,\dots,g_n),M_+][M_+,\tau(h_1,\dots,h_n),N_+] = [M_-,(\sigma\circ\tau)(g_1 h_1,\dots,g_n h_n),N_+] \text{,}
\]
which is again an element of $H_n(G)$. This shows that $H_n(G)$ is closed under products, and it also shows that the inverse of $[M_-,\sigma(g_1,\dots,g_n),M_+]$ is $[M_+,\sigma^{-1}(g_1^{-1},\dots,g_n^{-1}),M_-]$, which is again in $H_n(G)$, so $H_n(G)$ is a group.
\end{proof}

\section{Properties of Houghton-like groups}\label{sec:properties}

In this section we study some properties of the Houghton-like group $H_n(G)$ of a shift-similar group $G$. Our first goal is to inspect the relationship between $H_n(G)$ and $G$, then we will turn to issues of amenability and related properties, and finally we will discuss finite generation, and questions of higher finiteness properties (where very little is known).

\subsection{Relationship between a Houghton-like group and its shift-similar group}\label{ssec:relationship}

The first, somewhat obvious relationship is that $G$ embeds into $H_n(G)$, and in fact $G^n$ embeds into $H_n(G)$.

\begin{observation}\label{obs:G_embed}
For any shift-similar group $G\le\Sym(\N)$ and any $n\in\N$, the map $(g_1,\dots,g_n)\mapsto [\emptyset,(g_1,\dots,g_n),\emptyset]$ is an injective homomorphism $G^n\to H_n(G)$.
\end{observation}

\begin{proof}
This map is obviously a homomorphism, so we just need to check that it is injective. If $[\emptyset,(g_1,\dots,g_n)),\emptyset]$ equals the identity, then for each $1\le k\le n$ the restriction of this element to the quasi-ray $Q(k,\emptyset)$ shows that $g_k=1$.
\end{proof}

We can also dispense with any questions about the finite case easily:

\begin{observation}\label{obs:finite_boring}
For $n\in\N$ and $G\le \Sym(\N)$ a finite shift-similar group, we have $H_n(G)=H_n$.
\end{observation}

\begin{proof}
By Lemma~\ref{lem:finite_shift_sim}, $G\le \Symfin(\N)$. Thus for any $g\in G$, there exists a finite composition of shifting maps sending $g$ to the identity. Thus, for any $[M_-,\sigma(g_1,\dots,g_n),M_+]\in H_n(G)$, up to expansions without loss of generality $g_1=\cdots=g_n=1$.
\end{proof}

The infinite $G$ case is, of course, much more interesting. Note that for shift-similar $G\le\Sym(\N)$, the ``first'' Houghton-like group $H_1(G)$ is also a subgroup of $\Sym(\N)$, containing $G$. In fact, the following shows that in the strongly shift-similar case, we have $H_1(G)=G$.

\begin{proposition}\label{prop:H1_equal}
Let $G\le\Sym(\N)$ be a strongly shift-similar group. Then $H_1(G)=G$.
\end{proposition}

\begin{proof}
We need to prove that $H_1(G)\le G$. Since $G$ is strongly shift-similar, we already know $G$ contains $H_1=\Symfin(\N)$ by Theorem~\ref{thrm:contain_symfin}. Let $[M_-,\sigma g,M_+]\in H_1(G)$, so $M_-,M_+\subseteq \N$ with $|M_-|=|M_+|<\infty$, $\sigma\colon M_+\to M_-$ is a bijection, and $g\in G$. Since $[M_+,\sigma^{-1},M_-]\in H_1\le G$, without loss of generality we have $M_-=M_+$, call it $M$, and $\sigma$ is the identity. Now we have $[M,g,M]$, and we need to show that it lies in $G$. If $M=\emptyset$ we are done, so assume $M\ne\emptyset$. Pick some $j\in M$. Since $G$ is strongly shift-similar, Lemma~\ref{lem:strong_shift1_surj} says that we can choose $g'\in \Stab_G(j)$ such that $g=\shift_j(g')$, and hence
\[
[M\setminus\{j\},g',M\setminus\{j\}] = [M,g,M] \text{.}
\]
Repeating this trick until $M$ is depleted, we end up with $[\emptyset,g'',\emptyset]=[M,g,M]$ for some $g''\in G$, so $[M,g,M]\in G$.
\end{proof}

\begin{remark}
For $G$ a finite shift-similar group, by Observation~\ref{obs:finite_boring} we have $H_1(G)=H_1=\Symfin(\N)$, which is not finitely generated. However, for $G\le\Sym(\N)$ any finitely generated, infinite, strongly shift-similar group (for example any of the uncountably many examples from Corollary~\ref{cor:uncountable}), thanks to Proposition~\ref{prop:H1_equal} we have that $H_1(G)=G$ is finitely generated. In particular $H_1(G)$ has better finiteness properties than $H_1$ when $G$ is strongly shift-similar. It is hard to tell whether to view this as unexpected behavior, in comparison to the ``Thompson world'' situation, since all the $V_d$ are finitely generated and even of type $\F_\infty$, so in passing from $V_d$ from $V_d(G)$ there is no chance of improving the finiteness properties anyway. However, there are situations where $V_d(G)$ has better finiteness properties than $G$ (e.g., when $G$ is the non-finitely presentable Grigorchuk group and $V_2(G)$ is the type $\F_\infty$ R\"over group \cite{belk16}), so perhaps it is not unexpected that $H_1(G)$ could have better finiteness properties than its constituent component $H_1$.
\end{remark}

Even if $G$ is shift-similar but not strongly shift-similar (the possibility of which is asked in Question~\ref{quest:not_strong}), so $H_1(G)$ and $G$ are different, we still get that $H_1(G)$ is shift-similar, and more generally any of the $H_n(G)$ admit shift-similar representations, as we now explain. In fact the $H_n(G)$ are all strongly shift-similar.

\begin{proposition}\label{prop:HnG_shift_sim}
Let $G\le\Sym(\N)$ be shift-similar. Then for any $n\in\N$, the Houghton-like group $H_n(G)$ admits a strongly shift-similar representation. Moreover, for any $m,n\in\N$ we have $H_m(H_n(G))\cong H_{mn}(G)$.
\end{proposition}

\begin{proof}
Let $\xi\colon [n]\times\N \to \N$ be the bijection
\[
\xi(k,i) = k+(i-1)n\text{,}
\]
and $\omega\colon \Sym(\N) \to \Sym([n]\times\N)$ the induced isomorphism $\omega(g)=\xi^{-1}\circ g\circ \xi$, both from the proof of Lemma~\ref{lem:even_per_houghton}. We claim that $\omega^{-1}(H_n(G))\le \Sym(\N)$ is strongly shift-similar. Since $H_n(G)$ contains $\Symfin([n]\times\N)$ and $\omega$ preserves size of support, $\omega^{-1}(H_n(G))$ contains $\Symfin(\N)$. Thus by Lemma~\ref{lem:stab_suffices}, it suffices to prove that for any element of $\omega^{-1}(H_n(G))$ taking $1$ to $1$, its image under $\shift_1$ is also in $\omega^{-1}(H_n(G))$.

We can conjugate by $\xi$ to rephrase all of this with respect to $[n]\times\N$ instead of $\N$. We will use the same notation $s_1$ and $\shift_1$ for the induced functions. Upon conjugating by $\xi$, the bijection $s_1\colon \N\to \N\setminus\{1\}$ defined by $s_1(i)=i+1$ becomes a bijection
\[
s_1 \colon [n]\times\N \to ([n]\times\N)\setminus\{(1,1)\}
\]
defined by
\[
s_1(k,i)=\left\{\begin{array}{ll}(k+1,i) &\text{ if } k<n \\
(1,i+1) &\text{ if } k=n\text{.}\end{array}\right.
\]
Our rephrased goal now is to show that for any $\eta\in H_n(G)$ fixing $(1,1)$, we have that
\[
\shift_1(\eta)\coloneqq s_1^{-1}\circ \eta|_{([n]\times\N)\setminus\{(1,1)\}} \circ s_1
\]
also lies in $H_n(G)$. For the sake of intuition, first consider the case when $\eta=[\emptyset,(g_1,\dots,g_n),\emptyset]$ (so our assumption means $g_1(1)=1$). Then it is clear that $\shift_1(\eta)=[\emptyset,(g_2,\dots,g_n,\shift_1(g_1)),\emptyset]$, which is indeed an element of $H_n(G)$. More generally, we see that for $\eta$ of the form $[M_-,\sigma(g_1,\dots,g_n),M_+]$, the element $\shift_1(\eta)$ is obtained by cyclically permuting the $g_k$, either applying a shifting map to $g_1$ (if $(1,1)\not\in M_+$) or not (if $(1,1)\in M_+$), and adjusting $M_+$ and $M_-$ appropriately. In particular we get another element of $H_n(G)$ as desired. This shows that $H_n(G)$ is shift-similar. The above discussion also makes it clear that it is strongly shift-similar, since for example we can cyclically permute the $g_k$ in the other direction and add a new (fixed) point to $M_+$ and $M_-$.

Finally, we need to explain why $H_m(H_n(G))\cong H_{mn}(G)$. This is most easily seen using partitions into finite subsets and quasi-rays. An element of $H_m(H_n(G))$ is given by partitioning $[m]\times\N$ into a finite subset $M_+$ and quasi-rays $Q(k,M_+)$ ($1\le k\le m$), doing it a second time using some $M_-$, mapping $M_+$ bijectively to $M_-$ via some $\sigma$, and mapping each $Q(k,M_+)$ bijectively to $Q(k,M_-)$ via some element $\eta_k\in H_n(G)$ composed with appropriate canonical bijections. Here we view $H_n(G)$ as a subgroup of $\Sym(\N)$ as above. Now for each $k$, use the canonical bijection $\beta_{k,M_+}$ and the bijection $\xi$ to identify $Q(k,M_+)$ with $[n]\times\N$, and similarly for $Q(k,M_-)$. Then we can describe $\eta_k$ by partitioning $Q(k,M_+)$ into a finite subset $N_+^k$ and quasi-rays $Q(\ell,N_+^k)$ ($1\le \ell\le n$), similarly partitioning $Q(k,M_-)$ into some $N_-^k$ and its quasi-rays, mapping $N_+^k$ to $N_-^k$ via some bijection, and mapping $Q(\ell,N_+^k)$ to $Q(\ell,N_-^k)$ via some element $g_\ell^k$ of $G$ composed with appropriate canonical bijections. All in all, this is equivalent to considering $([m]\times[n])\times\N$ and looking at all the self-bijections obtained by partitioning $([m]\times[n])\times\N$ into a finite subset $P_+$ and a ``quasi-ray'' $Q((k,\ell),P_+)$ for each $(k,\ell)\in [m]\times[n]$, doing it a second time with some $P_-$, mapping $P_+$ to $P_-$ via some bijection $\upsilon$, and sending $Q((k,\ell),P_+)$ to $Q((k,\ell),P_-)$ bijectively via $g_\ell^k$ composed with appropriate ``canonical bijections''. This is clearly isomorphic to $H_{mn}(G)$, as desired.
\end{proof}

In particular $H_1(G)$ is always strongly shift-similar, and can be viewed in some sense as a ``strengthening'' of $G$, if $G$ is shift-similar but not strongly shift-similar. Viewing $H_1(G)$ as a strongly shift-similar group, we can consider its group of germs at infinity $\Germs(H_1(G))$. The following shows that if $G$ does not equal $H_1(G)$ then in fact they differ quite a lot, by virtue of $\Germs(H_1(G))$ being much larger than $\Germs(G)$.

\begin{corollary}\label{cor:germ_union}
If $G\le \Symfin(\N)$ is infinite and shift-similar but not strongly shift-similar, then $\Germs(H_1(G))$ is a properly ascending directed union of copies of $\Germs(G)$.
\end{corollary}

\begin{proof}
Since $H_1(G)$ is strongly shift-similar, $\shift_\infty\colon \Germs(H_1(G))\to \Germs(H_1(G))$ is an automorphism. Since $G$ is not strongly shift-similar, the restriction of $\shift_\infty$ to $\Germs(G)\to \Germs(G)$ is not surjective, so $\Germs(G)$ is a proper subgroup of $\Germs(H_1(G))$. For any $\eta=[M_-,\sigma g,M_+]\in H_1(G)$, if $m$ is the maximum element of $M_+$, then $\shift_1^m(\eta)\in G$. Hence, every element of $\Germs(H_1(G))$ lies in the preimage of $\Germs(G)$ under some $\shift_\infty^m$. This shows that $\Germs(H_1(G))$ is the directed union of the $(\shift_\infty^m)^{-1}(\Germs(G))$, which are each isomorphic to $\Germs(G)$ since $\shift_\infty$ is an injective homomorphism. Finally, to see that these terms are properly ascending, i.e., that every $(\shift_\infty^m)^{-1}(\Germs(G))$ is a proper subgroup of $\Germs(H_1(G))$, note that $\shift_\infty^m$ is surjective on $\Germs(H_1(G))$, but $\Germs(G)\ne \Germs(H_1(G))$.
\end{proof}

This is a good point to remind the reader of Question~\ref{quest:not_strong}, which asks whether there actually exist infinite shift-similar groups that are not strongly shift-similar.

\subsection{Amenability and its relatives}\label{ssec:amenable}

In this subsection we inspect amenability and related properties, with the main result being that a strongly shift-similar $G$ is amenable if and only if $H_n(G)$ is. The related properties we consider are elementary amenability (which implies amenability) and the property of containing no non-abelian free subgroups (which is implied by amenability), and we will see that the analogous results hold for them as well. Note that for $\mathcal{P}$ any one of these three properties, any directed union of finite groups has property $\mathcal{P}$, any abelian group has property $\mathcal{P}$, and any extension of groups with property $\mathcal{P}$ has property $\mathcal{P}$. Let us also record the following useful slight generalization of this last fact:

\begin{observation}\label{obs:extend_amenable}
Let $\mathcal{P}$ be any of the properties of being elementary amenable, being amenable, or containing no non-abelian free subgroups. Let $G$ be a group generated by a subgroup $H$ together with a normal subgroup $N$. If $H$ and $N$ both satisfy property $\mathcal{P}$ then so does $G$.
\end{observation}

\begin{proof}
We have $G=HN$, so $G/N \cong H/(H\cap N)$. Since $H$ has property $\mathcal{P}$, and property $\mathcal{P}$ is inherited by quotients, this shows that $G/N$ has property $\mathcal{P}$. Since $N$ has property $\mathcal{P}$, and an extension of a group with property $\mathcal{P}$ by another group with property $\mathcal{P}$ has property $\mathcal{P}$, we conclude that $G$ has property $\mathcal{P}$.
\end{proof}

The Houghton group $H_n$ naturally surjects onto $\Z^{n-1}$, and it turns out this extends to any $H_n(G)$, as we now show. Specializing this coming proof to $G=\{1\}$ also provides a new way of viewing the map from $H_n$ onto $\Z^{n-1}$ in terms of representative triples.

\begin{lemma}\label{lem:map_to_free_abelian}
Let $G$ be an infinite shift-similar group. The group $H_n(G)$ admits a map onto $\Z^{n-1}$, whose kernel is generated by $\Symfin([n]\times\N)$ and $H_1(G)^n$.
\end{lemma}

Here $H_1(G)^n$ means the natural copy of $H_1(G)^n$ inside of $H_n(G)$, given by viewing $[n]\times\N$ as the disjoint union of $n$ copies of $\N$ and taking $H_1(G)$ on each copy. Note that when $G=\{1\}$ we have $H_1(\{1\})^n=H_1^n\le \Symfin([n]\times\N)$, so this result recovers the fact that the kernel of $H_n\to\Z^{n-1}$ is $\Symfin([n]\times\N)$.

\begin{proof}[Proof of Lemma~\ref{lem:map_to_free_abelian}]
For each $1\le k\le n$ and each finite subset $M\subseteq [n]\times\N$, let $C_k(M)$ denote the cardinality of the intersection of $M$ with $\{k\}\times\N$. Define a map
\[
\chi_k([M_-,\sigma(g_1,\dots,g_n),M_+]) \coloneqq C_k(M_-) - C_k(M_+) \text{.}
\]
This is well defined up to equivalence, and yields a surjective group homomorphism $H_n(G)\to \Z$. Intuitively, when $G=\{1\}$ the map $\chi_k$ extracts the translation length of the ``eventual translation'' of the $k$th ray. (For $G\ne\{1\}$ there is not as clear-cut of a way to interpret $\chi_k$ in terms of ``eventual'' behavior.) It is easy to see that $\chi_1,\dots,\chi_{n-1}$ are linearly independent, and that $\chi_1+\cdots+\chi_n=0$ since $|M_-|=|M_+|$. The kernel of the direct sum of $\chi_1,\dots,\chi_{n-1}$ consists of all $[M_-,\sigma(g_1,\dots,g_n),M_+]$ such that $C_k(M_-)=C_k(M_+)$ for all $k$. Given such an element, up to composing with an element of $\Symfin([n]\times\N)$ we can assume that $\sigma$ takes each $M_+\cap(\{k\}\times\N)$ bijectively to $M_-\cap(\{k\}\times\N)$, so our element lies in the natural copy of $H_1(G)^n$ inside $H_n(G)$.
\end{proof}

When $G=\{1\}$ and $n\ge 3$, the commutator subgroup of $H_n$ is precisely $\Symfin([n]\times\N)$ \cite{houghton79}, so the above map $H_n\to \Z^{n-1}$ is the abelianization map. (When $n=1,2$ the commutator subgroup of $H_n$ is the finitely supported alternating group instead.) For $G$ non-abelian, the commutator subgroup is larger than $\Symfin([n]\times\N)$, since for example it contains elements of the form $[\emptyset,(g_1,\dots,g_n),\emptyset]$ for $g_1,\dots,g_n\in [G,G]$. It is difficult for us to see any general statement about the commutator subgroup of $H_n(G)$, or the abelianization map, or the rank of the abelianization, since there seems to just be too much variety in the possibilities for $G$.

The map $H_n(G)\to \Z^{n-1}$, while not necessarily the abelianization map, does allow us to ``promote'' amenability from $G$ to $H_n(G)$, along with other, related properties, as we now explain.

\begin{theorem}\label{thrm:amenable}
Let $G\le\Sym(\N)$ be shift-similar. For $\mathcal{P}$ any of the properties of being elementary amenable, being amenable, or containing no non-abelian free subgroups, we have that $H_n(G)$ has property $\mathcal{P}$ if and only if $G$ has property $\mathcal{P}$.
\end{theorem}

\begin{proof}
The ``only if'' direction follows since $G$ embeds into $H_n(G)$, and any of these properties is inherited by subgroups.

Now we do the ``if'' direction, so we assume $G$ has property $\mathcal{P}$ and want to show that $H_n(G)$ does too. First suppose $G$ is strongly shift-similar, so $G=H_1(G)$ by Proposition~\ref{prop:H1_equal} and thus $H_1(G)$ has property $\mathcal{P}$. Since property $\mathcal{P}$ is preserved under finite direct products, $H_1(G)^n$ has property $\mathcal{P}$. Since $\Symfin([n]\times\N)$ has property $\mathcal{P}$, Observation~\ref{obs:extend_amenable} and Lemma~\ref{lem:map_to_free_abelian} tell us that the kernel of $H_n(G)\to\Z^{n-1}$ has property $\mathcal{P}$. Finally, since $\Z^{n-1}$ has property $\mathcal{P}$ and an extension of groups with property $\mathcal{P}$ has property $\mathcal{P}$, we conclude that $H_n(G)$ has property $\mathcal{P}$. Now assume $G$ is infinite but not strongly shift-similar. By Corollary~\ref{cor:germ_union}, $\Germs(H_1(G))$ is a directed union of copies of $\Germs(G)$. Since property $\mathcal{P}$ is preserved under quotients, directed unions, and extensions by elementary amenable groups, we conclude that $H_1(G)$ has property $\mathcal{P}$. Now the same argument from the previous paragraph tells us that $H_n(G)$ has property $\mathcal{P}$. Finally, if $G$ is finite then by Observation~\ref{obs:finite_boring} $H_n(G)=H_n$, and so $H_n(G)$ has property $\mathcal{P}$.
\end{proof}

\subsection{Finite generation}\label{ssec:fin_gen}

Let $G\le \Sym(\N)$ be shift-similar and let $n\in\N$. Note that $H_n$ is a subgroup of $H_n(G)$, given by all elements of the form $[M_-,\sigma(1,\dots,1),M_+]$. Also, by Observation~\ref{obs:G_embed} we see that $G^n$ embeds as a subgroup of $H_n(G)$, consisting of all elements of the form $[\emptyset,(g_1,\dots,g_n),\emptyset]$. We know that $H_n$ is finitely generated for $n\ge 2$, and in this subsection we inspect finite generation of $H_n(G)$.

\begin{proposition}\label{prop:gens}
Let $G\le \Sym(\N)$ be shift-similar and let $n\ge 2$. Then the Houghton-like group $H_n(G)$ is generated by the copies of $H_n$ and $G^n$ referenced above.
\end{proposition}

\begin{proof}
Given an arbitrary element $[M_-,\sigma(g_1,\dots,g_n),M_+]$, left multiplication by some $[M_+,\tau,M_-]\in H_n$ lets us assume without loss of generality that $M_-=M_+$ and $\sigma=\id$. Thus, our element is of the form $[M,(g_1,\dots,g_n),M]$. This clearly equals
\[
[M,(g_1,1,\dots,1),M][M,(1,g_2,1,\dots,1),M]\cdots[M,(1,\dots,1,g_n),M]\text{,}
\]
so up to using parallel arguments, without loss of generality our element is of the form $[M,(g,1,\dots,1),M]$. Up to general reductions, we can also assume that $M\subseteq \{1\}\times\N$, and up to multiplication by elements of $H_n$ we can assume that $M=\{(1,1),\dots,(1,m)\}$ for some $m\in\N$. Let $M'=\{(2,1),\dots,(2,m)\}$ (here we are using that $n\ge 2$). Let $\mu\colon M\to M'$ be the bijection $(1,i)\mapsto (2,i)$, so $[M',\mu,M]\in H_n$. Upon conjugating our element $[M,(g,1,\dots,1),M]$ by $[M',\mu,M]$, we get $[M',\mu,M][M,(g,1,\dots,1),M][M,\mu^{-1},M'] = [M',(g,1,\dots,1),M']$, and performing reductions this equals $[\emptyset,(g,1,\dots,1),\emptyset]\in G^n$, so we are done.
\end{proof}

\begin{corollary}\label{cor:fin_gen}
For $n\ge 2$, if $G$ is finitely generated then so is $H_n(G)$.
\end{corollary}

\begin{proof}
By Proposition~\ref{prop:gens}, $H_n(G)$ is generated by $H_n$ and $G^n$, which are both finitely generated.
\end{proof}

When $n=1$, we have the following:

\begin{observation}\label{obs:H1_fin_gen}
Let $G\le \Sym(\N)$ be shift-similar and finitely generated. Then $H_1(G)$ is finitely generated if and only if $G$ is strongly shift-similar.
\end{observation}

\begin{proof}
If $G$ is strongly shift-similar then $H_1(G)=G$ by Proposition~\ref{prop:H1_equal}, so this is immediate. If $G$ is not strongly shift-similar then Corollary~\ref{cor:germ_union} says $\Germs(H_1(G))$ is a properly ascending directed union, and so cannot be finitely generated. Since $\Germs(H_1(G))$ is a quotient of $H_1(G)$, the latter cannot be finitely generated either.
\end{proof}

We should reiterate Question~\ref{quest:not_strong}, which asks whether shift-similar but not strongly shift-similar groups even exist, since a ``no'' answer would render Observation~\ref{obs:H1_fin_gen} uninteresting.

An obvious question regards higher finiteness properties. Since $H_n$ is of type $\F_{n-1}$ by Brown \cite{brown87}, one naturally conjectures that if a shift-similar group $G$ is of type $\F_{n-1}$, then so is $H_n(G)$. Moreover, since for strongly shift-similar $G$ we have $H_1(G)=G$, one might conjecture for $n\ge 2$ as well that $H_n(G)$ could sometimes have stronger finiteness properties than $H_n$ (assuming $G$ does). However, beyond finite generation it is surprisingly difficult to say anything about higher finiteness properties of $H_n(G)$.

For the experts, one can construct an analog of the contractible cube complex for $H_n$ from \cite{brown87,lee12}, inspired by the complexes for R\"over--Nekrashevych groups from \cite{belk16,skipper21}. The vertex stabilizers are copies of $G^n$, but for $G$ infinite we lose many nice local finiteness properties, for example the analog of \cite[Lemma~3.6]{skipper21} fails, and it is not clear that things can be set up to apply Brown's Criterion. Additionally, the resulting descending links are much more difficult than in the $G=\{1\}$ case. All in all, this makes it hard to tell how to say anything about higher finiteness properties for $H_n(G)$. Finally, it turns out to even be difficult to begin with to find examples of shift-similar groups with stronger finiteness properties than finite generation. In other words, it is not clear whether it would actually be interesting to be able to leverage, say, finite presentability of $G$ to get finite presentability of $H_n(G)$, given that we have so few examples of finitely presented shift-similar groups anyway.

Let us boil this discussion down the following two questions:

\begin{question}\label{quest:inherit_fin_props}
If a shift-similar group $G$ is finitely presented and $n\ge 3$, then is $H_n(G)$ finitely presented? If $G$ is of type $\F_m$ and $n\ge m+1$, then is $H_n(G)$ of type $\F_m$? If $G$ is strongly shift-similar, then do these results hold for all $n$?
\end{question}

For example, $H_n(H_k)\cong H_{nk}$ by Proposition~\ref{prop:HnG_shift_sim}, so this question has a positive answer when $G=H_k$. It also has a positive answer for $G=E_k$, since it is easy to check that $H_n(E_k)$ contains $H_n(H_k)$ with finite index.

\begin{question}\label{quest:examples_fin_props}
What sorts of finitely presented groups are shift-similar, beyond $E_n$ and $H_n$ for $n\ge 3$? What about shift-similar groups of type $\F_m$, beyond $E_n$ and $H_n$ for $n\ge m+1$?
\end{question}

\bibliographystyle{alpha}

\end{document}